\documentclass[final,1p,times]{elsarticle}

\usepackage[colorlinks=true]{hyperref}
\usepackage{amsfonts}
\usepackage[all]{xy}
\usepackage[mathscr]{eucal}
\usepackage{extarrows}
\usepackage{amsthm,amsmath,amssymb,mathrsfs,amsfonts,graphicx,graphics,latexsym,exscale,cmmib57,dsfont,bbm,amscd}
\usepackage{color,soul}
\newtheorem{thm}{Theorem}[section]

\newtheorem{lem}[thm]{Lemma}
\newtheorem{con}[thm]{Condition}

\newtheorem{prop}[thm]{Proposition}
\newtheorem{rem}[thm]{Remark}
\theoremstyle{definition}
\newtheorem{defn}[thm]{Definition}

\numberwithin{equation}{section}
\journal{xxx}

\begin{document}
\begin{sloppypar}

\begin{frontmatter}
\title{Scaled packing pressures on subsets for amenable group actions}

\author{Zubiao Xiao\corref{cor1}}
\ead{xzb2020@fzu.edu.cn}
\address{School of Mathematics and Statistics, Fuzhou University, Fuzhou 350116, People's Republic of China}
\author{Hongwei Jia}
\ead{jiahongwei2878@163.com}
\address{School of Mathematics and Statistics, Fuzhou University, Fuzhou 350116, People's Republic of China}
\author{Zhengyu Yin}
\ead{yzy\_nju\_20@163.com}
\address{Department of Mathematics, Nanjing University, Nanjing 210093, People's Republic of China}

\begin{abstract}
In this paper, we study the properties of the scaled packing topological pressures for topological dynamical system $(X,G)$, where $G$ is a countable discrete infinite amenable group.  We show that the scaled packing topological pressures can be determined by the scaled Bowen topological pressures. We obtain Billingsley's Theorem for the scaled packing pressures with a $G$-action. Then we get a variational principle between the scaled packing pressures and the scaled measure-theoretic upper local pressures. Finally, we give some restrictions on the scaled sequence $\mathbf{b}$, then in the case of the set $X_{\mu}$ of generic points, we prove that
 $$P^{P}(X_{\mu},\left\{F_{n}\right\},f,\mathbf{b})=h_{\mu}(X)+\int_{X} f \mathrm{d}\mu,$$
 if $\left\{F_{n}\right\}$ is tempered and $\mu$ is a $G$-invariant ergodic Borel probability measure.
\end{abstract}

\end{frontmatter}

\section{Introduction}\label{sec1}
Let $(X, G)$ be a topological dynamical system (TDS for short), where $X$ is a compact metric space with a metric $d$ and $G$ is a countable discrete infinite amenable group acting continuously on $X$. We write $(X,G)$ as $(X,T)$, when $G=\mathbb{Z}$ (or $G=\mathbb{Z}_{+}$) and $T:X \to X$ is a continuous map.

In 1965, Adler et al. \cite{Adl} introduced the topological entropy for a TDS $(X,T)$. Later, in 1973, Bowen \cite{Bo} introduced a definition of the topological entropy on subsets for $(X,T)$, which is well known as the \textit{Bowen topological entropy}. The Bowen topological entropy on the whole space coincides with the Adler-Konheim-McAndrew topological entropy. It can be seen as being dynamically analogous to the Hausdorff dimension, and plays a key role in the connection with the dimension theory, statistical physics and multifractal analysis (see, for example, \cite{Pe1}). A similar concept, the \textit{packing topological entropy}, in dynamical systems with other forms of dimensions was introduced by Feng and Huang \cite{Fe}. They proved a variational principle between the Bowen topological entropy and the packing topological entropy. Dou et al. \cite{Dou} studied the packing topological entropy in the framework of a group action.

Topological pressure can be regarded as a generalization of topological entropy, which first introduced by Ruelle \cite{Ru} and extended by Walters \cite{Wa} on compact spaces with continuous transformations. Pesin and Pitskel \cite{Pe2} generalized Bowen entropy and introduced a topological pressure with dimensional characteristics, which is called the \textit{Pesin-Pitskel topological pressure}. Zhong and Chen \cite{Zh} investigated a variational principle for the packing pressures when $G=\mathbb{Z}$. Recently, Ding et al. \cite{Ding} generalized the results of Dou et al. \cite{Dou} to the packing topological pressure.

A more appropriate scale is meaningful for us to study systems with sub-exponential  divergence of orbits. Zhao and Pesin \cite{Zhao} introduced a scaled entropy to investigate the complexity of the system with zero or infinite topological entropy. In \cite{Ch}, Chen and Li introduced the scaled packing topological entropy for amenable group actions and establish a corresponding variational principle.

In this paper, we continue this line of the research based on the above works. In section \ref{sec2}, we provide the definitions of the scaled topological pressures for a TDS with a $G$-action, and prove that the scaled packing topological pressure can be determined by the scaled Bowen topological pressure (Theorem \ref{T1}). In section \ref{sec3}, we study the relationship of various scaled measure-theoretic pressures for a $G$-action TDS (Theorem \ref{T5} and Theorem \ref{T2}). In section \ref{sec4}, we give a variational principle between the scaled packing pressure and the scaled measure-theoretic upper local pressure (Theorem \ref{T3}). In section \ref{sec5}, we recall the notion of a generic point and study a variational principle of the scaled packing topological pressure for the set of generic points with a $G$-action (Theorem \ref{T4}).
\section{Scaled topological pressures for a TDS with a $G$-action} \label{sec2}

\subsection{Amenable groups and scaled sequences}

Let $F(G)$ be the collection of nonempty finite subsets of $G$. A countable discrete infinite group $G$ is $\it amenable$ if there is a sequence $\left\{F_{n}\right\}\subseteq F(G)$ such that
$$\lim_{n \to+ \infty} \frac{\left | F_{n} \bigtriangleup gF_{n} \right |}{\left | F_{n} \right | }=0, \ \forall g\in G.$$
Such a sequence is called a \textit{F{\o}lner sequence}. A F{\o}lner sequence $\left\{F_{n}\right\}$ in $G$ is said to be \textit{tempered} if there exists a constant $C > 0$ which is independent of $n$ such that
$$\left | \bigcup_{k<n} F_{k}^{-1}F_{n}  \right | \leq C\left | F_{n}  \right |,\  \forall n \in \mathbb{N}.  $$

Let $F,A \in F(G)$ and $\varepsilon >0$. We say that $A$ is $(F,\varepsilon)$-\textit{invariant} if
$$|\left\{s\in A : Fs\subset A\right\}| \geq (1-\varepsilon)|A|.$$
The $F$ boundary of $A$ is defined by
$$\partial_{F}A:=\{c \in G :Fc \cap A \neq \emptyset , Fc \cap (G\backslash A) \neq \emptyset\}. $$
We note that a set $A$ is $(F,\varepsilon)$-invariant if and only if $\frac {|\partial_{F}A|}{|A|}<\varepsilon$ (see \cite[Remark 2.7]{Li}).
Moreover, we can get an observation as followed.
\begin{prop} \cite[Proposition 2.4]{Li}  \label{p6}
   Let $G$ be an amenable group. The sequence $\{F_{n}\} \subseteq F(G)$ is F{\o}lner if and only if for every finite set $F \subset G$ and $\varepsilon >0$, there is an $N\in \mathbb{N}$ such that $F_{n}$ is $(F,\varepsilon)$-invariant for all $n \geq N$.
\end{prop}
A positive sequence $\mathbf{b}=\left\{b(n)\right\}_{n\in \mathbb{N}}$ is called a \textit{scaled sequence}, if it monotonically increases to infinity. We denote by $\mathcal{SS}$ the set of all scaled sequences.
\subsection{The constructions of the scaled topological pressures}
Through this subsection, we will introduce the definitions of the scaled topological pressures.  We always assume that $\left\{F_{n}\right\}$ is a strictly monotonically increasing F{\o}lner sequence and $\mathbf{b}$ is a scaled sequence.

For any $F\in F(G)$, we define
$$d_{F}(x,y)=\max_{g \in F} d(gx,gy),\ \forall x,y \in X.$$
It is easy to see that $d_{F}$ is a metric on $X$.
For $\varepsilon>0$ and $x \in X$, we denote
$$B_{F}(x,\varepsilon )=\left\{y \in X: d_{F}(x,y) <\varepsilon \right\}$$
and
$$\overline{B}_{F}(x,\varepsilon )=\left\{y \in X: d_{F}(x,y) \leq \varepsilon \right\},$$
which are respectively the \textit{open and closed Bowen balls} with the center $x$ and the radius $\varepsilon$.

Let $f \in C(X,\mathbb{R})$, where $C(X,\mathbb{R})$ denote the set of all continuous functions of $X$. For any $F \in F(G)$, we write
$$\begin{aligned}
   f_{F}(x)&=\sum_{g \in F}f(g(x)),\\
   f_{F}(x,\varepsilon )&=\sup_{y \in B_{F}(x,\varepsilon )}f_{F}(y),\\
   \overline{f}_{F}(x,\varepsilon )&=\sup_{y \in \overline{B}_{F}(x,\varepsilon )}f_{F}(y).
\end{aligned}$$

For any $Z\subset X$, $N \in \mathbb{N}$, $\alpha \in \mathbb{R}$, $\varepsilon > 0$, $\mathbf{b}\in \mathcal{SS}$, and $f \in C(X,\mathbb{R})$, we define
\begin{equation}
M(N,\alpha ,\varepsilon ,Z,\left\{F_{n}\right\},f,\mathbf{b})= \inf\left \{ \sum_{i} e^{-\alpha b( | F_{n_{i} } |)+f_{F_{n_{i} } } (x_{i}) }  \right \}, \label{eq:2.1}
\end{equation}
where the infimum is taken over all finite or countable collections of $\left\{B_{F_{n_{i}}} (x_{i},\varepsilon )\right\}_{i}$, such that $x_{i} \in X$, $n_{i}\geq N$ for all $i$, and $Z \subseteq \bigcup_{i}B_{F_{n_{i}}}(x_{i},\varepsilon)$. Similarly, we define
\begin{equation}
R(N,\alpha ,\varepsilon ,Z,\left\{F_{n}\right\},f,\mathbf{b})= \inf\left \{ \sum_{i} e^{-\alpha b( | F_{N}|)+f_{F_{N} } (x_{i}) }  \right \}, \label{eq:2.2}
\end{equation}
where the infimum is taken over all finite or countable collections of $\left\{B_{F_{N}} (x_{i},\varepsilon )\right\}_{i}$, such that $x_{i} \in X$ for all $i$, and $Z \subseteq \bigcup_{i}B_{F_{N}}(x_{i},\varepsilon)$. Define
\begin{equation}
M^{P}(N,\alpha ,\varepsilon ,Z,\left\{F_{n}\right\},f,\mathbf{b})= \sup\ \left \{ \sum_{i} e^{-\alpha b( | F_{n_{i} }|)+f_{F_{n_{i} } } (x_{i}) }  \right \}, \label{eq:2.3}
\end{equation}
where the supremum is taken over all finite or countable pairwise disjoint collections of $\left\{\overline{B}_{F_{n_{i}}} (x_{i},\varepsilon )\right\}_{i}$, such that $x_{i} \in Z$, $n_{i}\geq N$ for all $i$. The quantity $M^{P}(N,\alpha ,\varepsilon ,Z,\left\{F_{n}\right\},f,\mathbf{b})$ does not increase as $N$ increases, hence the following limit exists:
$$M^{P}(\alpha ,\varepsilon ,Z,\left\{F_{n}\right\},f,\mathbf{b})=\lim_{N\to +\infty}M^{P}(N,\alpha ,\varepsilon ,Z,\left\{F_{n}\right\},f,\mathbf{b}).$$
Similarly, we can define
$$M(\alpha ,\varepsilon ,Z,\left\{F_{n}\right\},f,\mathbf{b})=\lim_{N\to+\infty}M(N,\alpha ,\varepsilon ,Z,\left\{F_{n}\right\},f,\mathbf{b}),$$
$$\overline{R}(\alpha ,\varepsilon ,Z,\left\{F_{n}\right\},f,\mathbf{b})=\limsup_{N\to+\infty}R(N,\alpha ,\varepsilon ,Z,\left\{F_{n}\right\},f,\mathbf{b}),$$
$$\underline{R}(\alpha ,\varepsilon ,Z,\left\{F_{n}\right\},f,\mathbf{b})=\liminf_{N\to+\infty}R(N,\alpha ,\varepsilon ,Z,\left\{F_{n}\right\},f,\mathbf{b}).$$
Define
$$M^{\mathcal{P} }(\alpha ,\varepsilon ,Z,\left\{F_{n}\right\},f,\mathbf{b})=\inf\left \{ \sum_{i=1}^{\infty } M^{P}(\alpha ,\varepsilon ,Z_{i},\left\{F_{n}\right\},f,\mathbf{b}) :Z\subseteq \bigcup_{i=1}^{\infty} Z_{i}\right \}.$$
It is easy to check that there exists a critical value of the parameter $\alpha$ , when $\alpha$ goes from $-\infty$ to $+\infty$, such that the quantities
$$M(\alpha ,\varepsilon ,Z,\left\{F_{n}\right\},f,\mathbf{b}),\ \overline{R}(\alpha ,\varepsilon ,Z,\left\{F_{n}\right\},f,\mathbf{b}),\ \underline{R}(\alpha ,\varepsilon ,Z,\left\{F_{n}\right\},f,\mathbf{b}),\ M^{\mathcal{P} }(\alpha ,\varepsilon ,Z,\left\{F_{n}\right\},f,\mathbf{b})$$ jump from $+\infty$ to $0$ respectively. Hence we can define the values

$$\begin{aligned}
P^{B}(\varepsilon ,Z,\left\{F_{n}\right\},f,\mathbf{b})
&=\sup\left\{\alpha:M(\alpha ,\varepsilon ,Z,\left\{F_{n}\right\},f,\mathbf{b})=+\infty\right\}\\
&=\inf\left\{\alpha:M(\alpha ,\varepsilon ,Z,\left\{F_{n}\right\},f,\mathbf{b})=0\right\},\\
\overline{CP}(\varepsilon ,Z,\left\{F_{n}\right\},f,\mathbf{b})
&=\sup\left\{\alpha:\overline{R}(\alpha ,\varepsilon ,Z,\left\{F_{n}\right\},f,\mathbf{b})=+\infty\right\}\\
&=\inf\left\{\alpha:\overline{R}(\alpha ,\varepsilon ,Z,\left\{F_{n}\right\},f,\mathbf{b})=0\right\},\\
\underline{CP}(\varepsilon ,Z,\left\{F_{n}\right\},f,\mathbf{b})
&=\sup\left\{\alpha:\underline{R}(\alpha ,\varepsilon ,Z,\left\{F_{n}\right\},f,\mathbf{b})=+\infty\right\}\\
&=\inf\left\{\alpha:\underline{R}(\alpha ,\varepsilon ,Z,\left\{F_{n}\right\},f,\mathbf{b})=0\right\},\\
P^{P}(\varepsilon ,Z,\left\{F_{n}\right\},f,\mathbf{b})
&=\sup\left\{\alpha:M^{\mathcal{P} }(\alpha ,\varepsilon ,Z,\left\{F_{n}\right\},f,\mathbf{b})=+\infty\right\}\\
&=\inf\left\{\alpha:M^{\mathcal{P} }(\alpha ,\varepsilon ,Z,\left\{F_{n}\right\},f,\mathbf{b})=0\right\}.
\end{aligned}$$
It is not hard to see that $$P^{B}(\varepsilon ,Z,\left\{F_{n}\right\},f,\mathbf{b}), \ \overline{CP}(\varepsilon ,Z,\left\{F_{n}\right\},f,\mathbf{b}), \ \underline{CP}(\varepsilon ,Z,\left\{F_{n}\right\},f,\mathbf{b}),\ P^{P}(\varepsilon ,Z,\left\{F_{n}\right\},f,\mathbf{b})$$ increases respectively when $\varepsilon$ decreases.
\begin{defn}
   We call the following quantities
   $$\begin{aligned}
      P^{B}(Z,\left\{F_{n}\right\},f,\mathbf{b})&=\lim_{\varepsilon \to 0} P^{B}(\varepsilon ,Z,\left\{F_{n}\right\},f,\mathbf{b}),\\
      \overline{CP}(Z,\left\{F_{n}\right\},f,\mathbf{b})&=\lim_{\varepsilon \to 0} \overline{CP}(\varepsilon ,Z,\left\{F_{n}\right\},f,\mathbf{b}),\\
      \underline{CP}(Z,\left\{F_{n}\right\},f,\mathbf{b})&=\lim_{\varepsilon \to 0} \underline{CP}(\varepsilon ,Z,\left\{F_{n}\right\},f,\mathbf{b}),\\
      P^{P}(Z,\left\{F_{n}\right\},f,\mathbf{b})&=\lim_{\varepsilon \to 0} P^{P}(\varepsilon ,Z,\left\{F_{n}\right\},f,\mathbf{b}),
      \end{aligned}$$
      \textit{scaled Pesin-Pitskel, scaled lower capacity, scaled upper capacity}, and \textit{scaled packing topological pressures} on the subset $Z$ with respect to $f$, $\mathbf{b}$ and $\left\{F_{n}\right\}$.
\end{defn}

Replacing $f_{F_{n_{i} } } (x_{i})$ in equations \eqref{eq:2.1} and \eqref{eq:2.3} by $f_{F_{n_{i} } } (x_{i},\varepsilon )$ and $\overline{f}_{F_{n_{i} } } (x_{i},\varepsilon )$ respectively and  $f_{F_{N} } (x_{i})$ in equation \eqref{eq:2.2} by $f_{F_{N} } (x_{i},\varepsilon)$, we can define new functions $\mathcal{M}$ , $\mathcal{R}$ , $\mathcal{M} ^{P}$. Then we denote the respective critical values by
$$P^{B'}(\varepsilon ,Z,\left\{F_{n}\right\},f,\mathbf{b}), \ \overline{CP'}(\varepsilon ,Z,\left\{F_{n}\right\},f,\mathbf{b}), \ \underline{CP'}(\varepsilon ,Z,\left\{F_{n}\right\},f,\mathbf{b}), \ P^{P'}(\varepsilon ,Z,\left\{F_{n}\right\},f,\mathbf{b}).$$

\begin{con}\label{con1}
   Let $\mathbf{b} \in \mathcal{SS}$ and $\left\{F_{n}\right\}$ be a F{\o}lner sequence in $G$. For all $n \in \mathbb{N}$,
   $$b(|F_{n}|) \geq |F_{n}|.$$
\end{con}

\begin{prop}\label{p1}
Let $(X,G)$ be a TDS with a $G$-action and a F{\o}lner  sequence $\left\{F_{n}\right\}$ satisfying Condition \ref{con1}. If $f \in C(X,\mathbb{R})$, $Z \subseteq X$, then we have
$$\begin{aligned}
   P^{B}(Z,\left\{F_{n}\right\},f,\mathbf{b})&=\lim_{\varepsilon \to 0} P^{B'}(\varepsilon ,Z,\left\{F_{n}\right\},f,\mathbf{b}),\\
   \overline{CP}(Z,\left\{F_{n}\right\},f,\mathbf{b})&=\lim_{\varepsilon \to 0} \overline{CP'}(\varepsilon ,Z,\left\{F_{n}\right\},f,\mathbf{b}),\\
   \underline{CP}(Z,\left\{F_{n}\right\},f,\mathbf{b})&=\lim_{\varepsilon \to 0} \underline{CP'}(\varepsilon ,Z,\left\{F_{n}\right\},f,\mathbf{b}),\\
   P^{P}(Z,\left\{F_{n}\right\},f,\mathbf{b})&=\lim_{\varepsilon \to 0} P^{P'}(\varepsilon ,Z,\left\{F_{n}\right\},f,\mathbf{b}).
   \end{aligned}$$
\end{prop}
\begin{proof}
   We shall prove $P^{P}(Z,\left\{F_{n}\right\},f,\mathbf{b})=\lim_{\varepsilon \to 0} P^{P'}(\varepsilon ,Z,\left\{F_{n}\right\},f,\mathbf{b})$. The other equalities can be proven similarly.
   It is clear that
   $$P^{P}(Z,\left\{F_{n}\right\},f,\mathbf{b})\leq\lim_{\varepsilon \to 0} P^{P'}(\varepsilon ,Z,\left\{F_{n}\right\},f,\mathbf{b}).$$
   We shall prove the converse inequality. Let
   $$\gamma (\varepsilon)=\sup\left\{\left|f(x)-f(y)\right|:d(x,y)\leq 2\varepsilon \right\}.$$
   For any $u,v \in \overline{B}_{F_{n_{i}}}(x,\varepsilon )$, and $g \in F_{n_{i}}$, we have $d(g(u),g(v)) \leq 2\varepsilon$, and then
   $$\left|f(g(u))-f(g(v))\right| \leq \gamma(\varepsilon).$$
   Since $\mathbf{b}\in \mathcal{SS}$ and $\left\{F_{n}\right\}$ satisfy Condition \ref{con1}, we have
   $$b(|F_{N}|)\geq |F_{N}|,\  \forall N\in \mathbb{N}.$$
   It follows that
   $$\begin{aligned}
   b(|F_{n_{i}}|) \gamma(\varepsilon) &\geq \sum_{g \in F_{n_{i}}}\left | f(g(u))-f(g(v)) \right | ,\\
   &\geq \overline{f}_{F_{n_{i} } } (x_{i},\varepsilon )- f_{F_{n_{i} } } (x_{i}), \  \forall n_{i}\geq N.
   \end{aligned}$$
   Then we have
   $$\begin{aligned}
   M^{P}(N,\alpha ,\varepsilon ,Z,\left\{F_{n}\right\},f,\mathbf{b})&= \sup\ \left \{ \sum_{i} e^{-\alpha b(| F_{n_{i} }|)+f_{F_{n_{i} } } (x_{i}) }  \right \}\\
   &\ge \sup\ \left \{ \sum_{i} e^{-(\alpha+\gamma(\varepsilon)) b(| F_{n_{i} }|)+\overline{f}_{F_{n_{i} } } (x_{i},\varepsilon) }\right \}\\
   &=\mathcal{M}^{P}(N,\alpha+\gamma(\varepsilon) ,\varepsilon ,Z,\left\{F_{n}\right\},f,\mathbf{b}).
\end{aligned}$$
It follows that $P^{P}(\varepsilon,Z,\left\{F_{n}\right\},f,\mathbf{b})\ge P^{P'}(\varepsilon ,Z,\left\{F_{n}\right\},f,\mathbf{b})-\gamma(\varepsilon)$. Since $f$ is uniformly continuous on $X$, we obtain the desired inequality by letting $\varepsilon \to 0$.
\end{proof}
Huang et al. \cite{Huang} gave another definition of the Pesin-Pitskel topological pressure. Now we introduce a scaled version.

For $K \in F(G)$ and $\delta>0$, we denote by $\mathcal{B}(K,\delta)$ the set of all $F\in F(G)$ satisfying $|KF \backslash F|<\delta|F|$.
The collection of pairs $\Lambda=\left\{(K,\delta):K\in F(G),\ \delta>0\right\} $  forms a net where $(K',\delta')\succ (K,\delta)$ means $K \subseteq K'$ and $\delta \geq \delta'$. For an $\mathbb{R}$-valued function $\varphi$ defined on $F(G)$, they define
$$\limsup_{F}\varphi (F):=\lim_{(K,\delta)\in \Lambda} \sup_{F\in \mathcal{B}(K,\delta)}\varphi(F),$$
and
$$\liminf_{F}\varphi (F):=\lim_{(K,\delta)\in \Lambda} \inf_{F\in \mathcal{B}(K,\delta)}\varphi(F).$$
By \cite[Remark 2.1]{Huang}, one can get that
 $$\limsup_{F}\varphi (F)=\inf_{(K,\delta)\in \Lambda} \sup_{F\in \mathcal{B}(K,\delta)}\varphi(F),$$
 $$\limsup_{F}\varphi (F)=\sup_{(K,\delta)\in \Lambda} \inf_{F\in \mathcal{B}(K,\delta)}\varphi(F).$$

Let $\varphi$ be a real-valued function on $F(G)$. $\varphi(F)$ \textit{converges to a limit $L$ as $F$ becomes more and more invariant} if for every $\varepsilon >0$,  there are a nonempty finite set $K \subset G$ and $\delta >0$ such that $|\varphi (F)-L<\varepsilon|$ for every $F \in \mathcal{B}(K,\delta)$. Note that if the limit $\lim_{F}\varphi(F)$ exists as $F$ becomes more and more invariant, then
$$\lim_{F}\varphi (F)=\limsup_{F}\varphi (F)=\liminf_{F}\varphi (F).$$
For any subset $Z \subseteq X$, $K \in F(G)$, $\delta >0$, and $s \in \mathbb{R}$, we can define
$$Q(K,\delta ,\alpha,\varepsilon,Z,f,\mathbf{b})= \inf_{\Gamma}\left \{ \sum_{i} e^{-\alpha b(| F_{i} |)+f_{F_{i} } (x_{i},\varepsilon) }  \right \},$$
where the infimum is taken over all finite of countable collections $\Gamma =\{B_{F_{i}}(x_{i},\varepsilon)\}_{i\in I}$ such that $F_{i}\in \mathcal{B}(K,\delta)$, $x_{i} \in X$ and $Z\subseteq \bigcup_{i}B_{F_{i}}(x_{i},\varepsilon)$.

The quantity $Q(K,\delta ,\varepsilon ,\alpha,Z,f,\mathbf{b})$ does not decrease as $(K,\delta)$ increases in the net $\Lambda$. Hence the following limit exists:
$$Q(\alpha,\varepsilon ,Z,f,\mathbf{b})=\lim_{(K,\delta) \in \Lambda}Q(K,\delta ,\alpha,\varepsilon,Z,f,\mathbf{b}).$$
\cite[Lemma 2.5]{Huang} says that there exists a critical value of $\alpha$, at which
$Q(\alpha,\varepsilon ,Z,f)$ jumps from $+\infty$ to $0$. Then we can define
$$Q(\varepsilon ,Z,f,\mathbf{b})=\inf \left\{\alpha : Q(\alpha,\varepsilon ,Z,f,\mathbf{b})=0 \right\}=\sup \left\{\alpha : Q(\alpha,\varepsilon ,Z,f,\mathbf{b})=+\infty \right\}.$$
We call
$$Q(Z,f,\mathbf{b})=\limsup_{\varepsilon \to 0}Q(\varepsilon ,Z,f,\mathbf{b})$$
the \textit{scaled Bowen Pesin-Pitskel pressure of} $Z$.

\begin{prop} \label{p7}
Let $(X,G)$ be a TDS with a $G$-action and a F{\o}lner  sequence $\left\{F_{n}\right\}$ satisfying Condition \ref{con1}. If $f \in C(X,\mathbb{R})$, $Z \subseteq X$, then we have
$$P^{B}(Z,\left\{F_{n}\right\},f,\mathbf{b}) = Q(Z,f,\mathbf{b}).$$
\end{prop}

\begin{proof}
By Proposition \ref{p1}, we shall prove $P^{B'}(Z,\left\{F_{n}\right\},f,\mathbf{b}) =Q(Z,f,\mathbf{b})$. Then the desired inequality follows.

\textit{Claim 1.} For every $\varepsilon>0$, if $s\in \mathbb{R}$ with $\mathcal{M}(s ,\varepsilon ,Z,\left\{F_{n}\right\},f,\mathbf{b})=0$, then $Q(s,\varepsilon ,Z,f,\mathbf{b})=0.$

Since $\mathcal{M}(s ,\varepsilon ,Z,\left\{F_{n}\right\},f,\mathbf{b})=0$, for every $\varepsilon_{0}>0$, there exists $N_{0}>0$ such that
$$\mathcal{M}(N,s ,\varepsilon ,Z,\left\{F_{n}\right\},f,\mathbf{b})<\varepsilon_{0}, \ \forall N\geq N_{0} .$$
Fix $N_1\geq N_0$. Then there exists finite or countable collections of $\left\{B_{F_{n_{i}}} (x_{i},\varepsilon )\right\}_{i}$, such that $x_{i} \in X$, $n_{i}\geq N_1$ for all $i$, $Z \subseteq \bigcup_{i}B_{F_{n_{i}}}(x_{i},\varepsilon)$ and
$$\sum_{i} e^{-s b(| F_{n_{i} } |)+f_{F_{n_{i} } } (x_{i},\varepsilon) }<\varepsilon_{0}. $$
By Proposition \ref{p6}, for every finite set $K \subset G$ and $\delta >0$, there exists $N_{2} \geq N_{1}$ such that $F_{n_i}$ is $(K,\delta)$-invariant for all $n_i \geq N_{2}$. Hence $F_{n_i} \in \mathcal{B}(K,\delta)$ for all $n_i \geq N_{2}$. It follows that
$$Q(K,\delta ,s,\varepsilon,Z,f,\mathbf{b}) < \varepsilon_{0}.$$
By the arbitrariness of $K$, $\delta$ and $\varepsilon_{0}$, we have
$$Q(s,\varepsilon,Z,f,\mathbf{b})=0.$$

\textit{Claim 2.} For every $\varepsilon>0$, if $t\in \mathbb{R}$ with $\mathcal{M}(t ,\varepsilon ,Z,\left\{F_{n}\right\},f,\mathbf{b})=+\infty$, then $Q(t,\varepsilon ,Z,f,\mathbf{b})=+\infty.$

Since $\mathcal{M}(t ,\varepsilon ,Z,\left\{F_{n}\right\},f,\mathbf{b})=+\infty$, for every $M>0$ and $N>0$, there exists $W_{0}>N$ such that
$$\mathcal{M}(W_{0},t,\varepsilon ,Z,\left\{F_{n}\right\},f,\mathbf{b})>M.$$
Then for any finite or countable collections of $\left\{B_{F_{n_{i}}} (x_{i},\varepsilon )\right\}_{i}$ with $x_{i} \in X$, $n_{i}\geq W_0$ for all $i$ and $Z \subseteq \bigcup_{i}B_{F_{n_{i}}}(x_{i},\varepsilon)$, we have
$$\sum_{i} e^{-t b(| F_{n_{i} } |)+f_{F_{n_{i} } } (x_{i},\varepsilon) }>M. $$
By Proposition \ref{p6}, for every finite set $K \subset G$ and $\delta >0$, there exists $W_{1} \geq W_{0}$ such that $F_{n_i}$ is $(K,\delta)$-invariant for all $n_i \geq W_{1}$. Hence $F_{n_i} \in \mathcal{B}(K,\delta)$ for all $n_i \geq W_{1}$. It follows that
$$Q(K,\delta ,t,\varepsilon,Z,f,\mathbf{b}) \geq M.$$
By the arbitrariness of $K$, $\delta$ and $\varepsilon_{0}$, we have
$$Q(t,\varepsilon,Z,f,\mathbf{b})=+\infty.$$

Hence by claim 1, we know that if $s>P^{B'}(\varepsilon,Z,\left\{F_{n}\right\},f,\mathbf{b})$, then $s>Q(\varepsilon,Z,f,\mathbf{b})$. By the arbitrariness of $s$, we have
\begin{equation} \label{q1}
P^{B'}(\varepsilon,Z,\left\{F_{n}\right\},f,\mathbf{b}) \geq Q(\varepsilon,Z,f,\mathbf{b}).
\end{equation}
By claim 2, we have that if $P^{B'}(\varepsilon,Z,\left\{F_{n}\right\},f,\mathbf{b})>t$, then $Q(\varepsilon,Z,f,\mathbf{b})>t$. By the arbitrariness of $t$, we have
\begin{equation}  \label{q2}
P^{B'}(\varepsilon,Z,\left\{F_{n}\right\},f,\mathbf{b}) \leq Q(\varepsilon,Z,f,\mathbf{b}).
\end{equation}
Adding $\limsup\limits_{\varepsilon\to 0}$ on the both hand of the equation \eqref{q1} and \eqref{q2}, we have
$$P^{B'}(Z,\left\{F_{n}\right\},f,\mathbf{b}) = Q(Z,f,\mathbf{b}).$$

\end{proof}

\subsection{The properties of the scaled topological pressures}
In this section, we will study the relationships among various scaled topological pressures and their properties.

\begin{prop}\label{p2}
   Let $(X,G)$ be a TDS with a $G$-action and a F{\o}lner  sequence $\left\{F_{n}\right\}$. If $f \in C(X,\mathbb{R})$ and $Z \subseteq X$, then we have
   \begin{enumerate}[(1)]
      \item $P^{B}(Z,\left\{F_{n}\right\},f,\mathbf{b}) \leq \underline{CP}(Z,\left\{F_{n}\right\},f,\mathbf{b}) \leq \overline{CP}(Z,\left\{F_{n}\right\},f,\mathbf{b}).$
      \item if $Z_{1} \subseteq Z_{2}$, then $\mathcal{P}(Z_{1},\left\{F_{n}\right\},f,\mathbf{b}) \leq \mathcal{P}(Z_{2},\left\{F_{n}\right\},f,\mathbf{b})$, where $\mathcal{P} \in \left\{P^{B},\underline{CP},\overline{CP},P^{P}\right\}.$ \label{p2.2}
      \item if $Z=\bigcup_{i \in I}Z_{i}$ with I at most countable, then
      \begin{enumerate}[(3-a)]
         \item $M(\alpha ,\varepsilon ,Z,\left\{F_{n}\right\},f,\mathbf{b}) \leq \sum_{i \in I}M(\alpha ,\varepsilon ,Z_{i},\left\{F_{n}\right\},f,\mathbf{b});$
         \item $M^{\mathcal{P}}(\alpha ,\varepsilon ,Z,\left\{F_{n}\right\},f,\mathbf{b}) \leq \sum_{i \in I}M^{\mathcal{P}}(\alpha ,\varepsilon ,Z_{i},\left\{F_{n}\right\},f,\mathbf{b});$
         \item $P^{B}(Z,\left\{F_{n}\right\},f,\mathbf{b})=\sup_{i \in I}P^{B}(Z_{i},\left\{F_{n}\right\},f,\mathbf{b});$
         \item $P^{P}(Z,\left\{F_{n}\right\},f,\mathbf{b})=\sup_{i \in I}P^{P}(Z_{i},\left\{F_{n}\right\},f,\mathbf{b}).$
      \end{enumerate}
   \end{enumerate}
\end{prop}
\begin{proof}
   (1) and (2) can be obtained directly by the definitions of the scaled topological pressures. We now show (3-b). Then (3-a) can be proven similarly. Given $\gamma >0$ and $i \in I $, we can find $\left\{Z_{i,j}\right\}_{j \geq 0}$ such that $Z_{i} \subseteq  \bigcup_{j \geq 0}Z_{i,j}$ and
   $$\sum_{j \geq 0}M^{{P}}(\alpha ,\varepsilon ,Z_{i,j},\left\{F_{n}\right\},f,\mathbf{b}) \leq M^{\mathcal{P}}(\alpha ,\varepsilon ,Z_{i},\left\{F_{n}\right\},f,\mathbf{b})+\frac{\gamma}{2^{i}}.$$ Thus
   $$\sum_{i \in I}\sum_{j \geq 0}M^{{P}}(\alpha ,\varepsilon ,Z_{i,j},\left\{F_{n}\right\},f,\mathbf{b}) \leq\sum_{i \in I} M^{\mathcal{P}}(\alpha ,\varepsilon ,Z_{i},\left\{F_{n}\right\},f,\mathbf{b})+2\gamma.$$
   Then it is obvious that
   $$M^{\mathcal{P}}(\alpha ,\varepsilon ,Z,\left\{F_{n}\right\},f,\mathbf{b}) \leq \sum_{i \in I}\sum_{j \geq 0}M^{{P}}(\alpha ,\varepsilon ,Z_{i,j},\left\{F_{n}\right\},f,\mathbf{b}).$$
   Letting $\gamma \to 0$, the desired inequality follows.

   We now show that (3-d) holds and (3-c) can be proven similarly. If $\sup_{i \in I}P^{P}(Z_{i},\left\{F_{n}\right\},f,\mathbf{b})<s$, then for any $\varepsilon>0$ and $i \in I$, $P^{P}(\varepsilon,Z_{i},\left\{F_{n}\right\},f,\mathbf{b})<s$, and thus $M^{\mathcal{P}}(s,\varepsilon ,Z_{i},\left\{F_{n}\right\},f,\mathbf{b})=0$, which implies $$M^{\mathcal{P}}(s,\varepsilon ,Z,\left\{F_{n}\right\},f,\mathbf{b})=0$$ (utilizing (3-b)). Hence $P^{P}(\varepsilon,Z,\left\{F_{n}\right\},f,\mathbf{b}) \leq s$. It follows that
   $$P^{P}(Z,\left\{F_{n}\right\},f,\mathbf{b}) \leq \sup_{i \in I}P^{P}(Z_{i},\left\{F_{n}\right\},f,\mathbf{b}).$$
   The opposite inequality follows from (\ref{p2.2}).
\end{proof}
\begin{lem}\label{l1}
    Let $(X,G)$ be a TDS with a $G$-action and a F{\o}lner  sequence $\left\{F_{n}\right\}$ satisfying Condition \ref{con1}. If $f \in C(X,\mathbb{R})$, $Z \subseteq X$, then we have
   $$P^{B}(Z,\left\{F_{n}\right\},f,\mathbf{b}) \leq P^{P}(Z,\left\{F_{n}\right\},f,\mathbf{b}) \leq \overline{CP}(Z,\left\{F_{n}\right\},f,\mathbf{b}).$$
\end{lem}
\begin{proof}
   We first show that
   $$P^{B}(Z,\left\{F_{n}\right\},f,\mathbf{b}) \leq P^{P}(Z,\left\{F_{n}\right\},f,\mathbf{b}).$$
   Suppose that  $P^{B}(Z,\left\{F_{n}\right\},f,\mathbf{b})>s>-\infty$. For any $ \varepsilon >0$ and $N \in \mathbb{N}$, we set
   $$\mathcal{F}_{N,\varepsilon}=\left\{ \mathcal{F}:x_{i} \in Z,\  \mathcal{F}=\left\{ \overline{B}_{F_{N}}(x_{i},\varepsilon)\right\} \text{is disjoint family} \right\}.$$
   Take $\mathcal{F}(N,\varepsilon,Z) \in \mathcal{F}_{N,\varepsilon}$ such that $\left| \mathcal{F}(N,\varepsilon,Z) \right|=\max_{\mathcal{F} \in \mathcal{F}_{N,\varepsilon}}\left \{ \left | \mathcal{F} \right |  \right \} $. We denote
   $$\mathcal{F}(N,\varepsilon,Z)=\left \{ \overline{B}_{F_{N}}(x_{i},\varepsilon ),i=1,2,\cdots ,\left | \mathcal{F}(N,\varepsilon ,Z) \right |   \right \}.$$
   It is not hard to see that
   $$Z\subseteq \bigcup_{i=1}^{\left | \mathcal{F}(N,\varepsilon ,Z) \right | }B_{F_{N}}(x_{i},2\varepsilon +\delta),\ \forall \delta >0.$$
   Then for any $s \in \mathbb{R}$, we have
   $$M(N,s,2\varepsilon + \delta ,Z,\left\{F_{n}\right\},f,\mathbf{b}) \leq e^{-b(|F_{N}|)s}\sum_{i=1}^{\left | \mathcal{F}(N,\varepsilon ,Z) \right | }e^{f_{F_{N}}(x_{i})} \leq M^{P}(N,s,\varepsilon ,Z,\left\{F_{n}\right\},f,\mathbf{b}) .$$
   Thus letting $N \to \infty$, we have
   $$M(s,2\varepsilon + \delta ,Z,\left\{F_{n}\right\},f,\mathbf{b}) \leq M^{P}(s,\varepsilon ,Z,\left\{F_{n}\right\},f,\mathbf{b}).$$
   Let $\left\{ Z_{i} \right\}_{i \in I}$ be a cover of $Z$. Then we have
   $$M(s,2\varepsilon + \delta ,Z_{i},\left\{F_{n}\right\},f,\mathbf{b}) \leq M^{P}(s,\varepsilon ,Z_{i},\left\{F_{n}\right\},f,\mathbf{b}).$$
   Thus
   $$\begin{aligned}
      M(s,2\varepsilon + \delta ,Z,\left\{F_{n}\right\},f,\mathbf{b}) &\leq \sum_{i \in I}M(s,2\varepsilon + \delta ,Z_{i},\left\{F_{n}\right\},f,\mathbf{b})\\
      &\leq  \sum_{i \in I}M^{P}(s,\varepsilon ,Z_{i},\left\{F_{n}\right\},f,\mathbf{b}).
   \end{aligned}$$
This implies $$M(s,2\varepsilon + \delta ,Z,\left\{F_{n}\right\},f,\mathbf{b}) \leq M^{\mathcal{P}}(s,\varepsilon ,Z,\left\{F_{n}\right\},f,\mathbf{b}).$$
Since $P^{B}(Z,\left\{F_{n}\right\},f,\mathbf{b})>s$, we have $M(s,2\varepsilon + \delta ,Z,\left\{F_{n}\right\},f,\mathbf{b}) \geq 1$, when $\varepsilon$ and $\delta$ are small enough. Then
$$M^{\mathcal{P}}(s,\varepsilon ,Z,\left\{F_{n}\right\},f,\mathbf{b}) \geq 1 .$$
Thus
$$P^{P}(\varepsilon,Z,\left\{F_{n}\right\},f,\mathbf{b}) \geq s.$$
Letting $\varepsilon \to 0$, we have $P^{P}(Z,\left\{F_{n}\right\},f,\mathbf{b}) \geq s$. Hence $P^{B}(Z,\left\{F_{n}\right\},f,\mathbf{b}) \leq P^{P}(Z,\left\{F_{n}\right\},f,\mathbf{b})$.

We shall show $P^{P}(Z,\left\{F_{n}\right\},f,\mathbf{b}) \leq \overline{CP}(Z,\left\{F_{n}\right\},f,\mathbf{b})$.

Choose $P^{P}(Z,\left\{F_{n}\right\},f,\mathbf{b})>s>t>-\infty$. Then there exists $\delta>0$ such that for any $\varepsilon \in (0,\delta)$,
$$P^{P}(\varepsilon,Z,\left\{F_{n}\right\},f,\mathbf{b}) \geq s$$
and
$$ M^{P}(s,\varepsilon ,Z,\left\{F_{n}\right\},f,\mathbf{b}) \geq  M^{\mathcal{P}}(s,\varepsilon ,Z,\left\{F_{n}\right\},f,\mathbf{b})=+\infty.$$
Since $\left\{F_{n}\right\}$ and $\mathbf{b}$ satisfy Condition \ref{con1}, $\sum_{k \geq 1}e^{b(|F_{k}|) (t-s)}$ converges. Then we can assume that
$$\sum_{k \geq 1}e^{b(|F_{k}|)(t-s)}=M.$$
Hence for any $ N \in \mathbb{N}$, there exists a countable pairwise disjoint family $\left\{\overline{B}_{F_{n_{i}}} (x_{i},\varepsilon )\right\}_{i}$ such that $x_{i} \in Z$, $n_{i} \geq N$ for all $i$ and $ \sum_{i} e^{-sb( | F_{n_{i} }|)+f_{F_{n_{i} } } (x_{i}) } >M$. For each $k$, let
$$m_{k}=\left \{ x _{i}:n_{i}=k\right \} .$$
Then
$$\sum_{k=N}^{\infty}\sum _{x \in m_{k}}e^{f_{F_{k}}(x)}e^{-b(|F_{k}|)s }>M.$$
It is not hard to see that there exists $k \geq N$ such that
$$\sum_{x \in m_{k}}e^{f_{F_{k}}(x)}e^{-b(|F_{k}|) t }  \geq 1-e^{t-s}.$$
Taking a collection $\left \{ B_{F_{k} } (y_{i},\frac{\varepsilon }{2} )\right \} _{i \in I}$, such that $Z \subseteq \bigcup_{i \in I }B_{F_{k}}(y_{i},\frac{\varepsilon}{2})$, one can obtain that for any $x_{1},x_{2} \in m_{k}$, there exist $y_{1}$, $ y_{2}$ with $y_{1} \neq y_{2}$ such that $x_{i} \in B_{F_{k}}(y_{i},\frac{\varepsilon}{2})$, $i=1,2$. Then
$$\mathcal{R}(k,t,\frac{\varepsilon}{2},Z,\left\{F_{n}\right\},f,\mathbf{b}) \geq \sum_{x \in m_{k}}e^{f_{F_{k}}(x)}e^{-b(|F_{k}|)t} \geq 1-e^{t-s}>0.$$
Thus $\overline{CP'}(\frac{\varepsilon}{2},Z,\left\{F_{n}\right\},f,\mathbf{b}) \geq t$. Letting $\varepsilon \to 0$, we can get that
$$\overline{CP}(Z,\left\{F_{n}\right\},f,\mathbf{b}) \geq t.$$
Hence $P^{P}(Z,\left\{F_{n}\right\},f,\mathbf{b}) \leq \overline{CP}(Z,\left\{F_{n}\right\},f,\mathbf{b})$.
\end{proof}
In \cite{Zh}, Zhong and Chen proved the following theorem when $G=\mathbb{Z}$ and $b(|F_{n}|)=|F_{n}|$ for all $n \in \mathbb{N}$. We shall show that the scaled packing topological pressure can be determined by the scaled Bowen topological pressure.
\begin{thm}\label{T1}
   Let $(X,G)$ be a TDS with a $G$-action and a F{\o}lner sequence $\left\{F_{n}\right\}$ satisfying Condition \ref{con1}. If $f \in C(X,\mathbb{R})$, $Z \subseteq X$ is compact and $G$-invariant, then
   $$P^{B}(Z,\left\{F_{n}\right\},f,\mathbf{b}) = P^{P}(Z,\left\{F_{n}\right\},f,\mathbf{b}) = \underline{CP}(Z,\left\{F_{n}\right\},f,\mathbf{b})= \overline{CP}(Z,\left\{F_{n}\right\},f,\mathbf{b}).$$
\end{thm}
\begin{proof}
   we shall prove $P^{B}(Z,\left\{F_{n}\right\},f,\mathbf{b}) \geq \overline{CP}(Z,\left\{F_{n}\right\},f,\mathbf{b})$. Then combining with Proposition \ref{p2} and Lemma \ref{l1}, the theorem is proved.

   Let $s>P^{B}(Z,\left\{F_{n}\right\},f,\mathbf{b})$. Then there exists $\varepsilon_{0}>0$ such that
   $$s>P^{B}(\varepsilon,Z,\left\{F_{n}\right\},f,\mathbf{b}),$$
   for every $\varepsilon \in (0,\varepsilon_{0})$. Hence $M(s,\varepsilon,Z,\left\{F_{n}\right\},f,\mathbf{b})=0$. This implies
   $$M(N,s,\varepsilon,Z,\left\{F_{n}\right\},f,\mathbf{b})=0,$$
   for any $N \in \mathbb{N}$. Then for any $\varepsilon \in (0,\varepsilon_{0})$, there exists $\left\{B_{F_{n_{i}}} (x_{i},\varepsilon )\right\}_{i}$ such that $Z \subseteq \bigcup_{i}B_{F_{n_{i}}} (x_{i},\varepsilon )$ and $\sum_{i} e^{-s b(| F_{n_{i} }|)+f_{F_{n_{i} } } (x_{i}) } < 1$, where $x_{i} \in Z$, $n_{i} \geq N$ for all $i$.

   We can assume that $\left\{B_{F_{n_{i}}} (x_{i},\varepsilon )\right\}_{i}$ is finite since Z is compact. Then there exists $k_{1},k_{2} \in \mathbb{N}$ such that $\left | F_{n_{i}} \right | \in \left [ \left | F_{N} \right | ,\left | F_{N} \right | +k_{1} \right ] $ and $b(|F_{n_{i}}|)\in [b(|F_{N}|),b(|F_{N}|)+k_{2}]$. It is obvious that $B_{F_{n_{i}}} (x_{i},\varepsilon ) \subseteq B_{F_{N}} (x_{i},\varepsilon )$. This implies that  $Z \subset \bigcup_{i}B_{F_{N}} (x_{i},\varepsilon )$ and $f_{F_{N}}(x_{i}) \leq f_{F_{n_{i}}}(x_{i})+k_{1}\left \| f \right \| $, where $\left \| f \right \| := \sup_{x \in Z}|f(x)|$. Hence
   $$\begin{aligned}
      \sum_{i}e^{-sb(| F_{N} |) +f_{F_{N}} (x_{i}) } & \leq e^{k_{1}\left \| f \right \| }\sum_{i}e^{-sb(| F_{n_{i}}|) +f_{F_{n_{i}}} (x_{i}) } e^{sb(| F_{n_{i}}|)-sb(| F_{N} |)  }\\
      & \leq e^{k_{1} \left \| f \right \|}\text{max}\left\{ e^{k_{2}s},1 \right\}<+\infty.
   \end{aligned}$$
    Thus $\overline{R}(s,\varepsilon,Z,\left\{F_{n}\right\},f,\mathbf{b})< +\infty$, and then $s \geq \overline{CP}(Z,\left\{F_{n}\right\},f,\mathbf{b})$.
\end{proof}
\section{Scaled measure-theoretic pressures}\label{sec3}
We denote the set of all Borel probability measures, $G$-invariant Borel probability measures and $G$-invariant ergodic Borel probability measures on $X$ by $M(X)$, $M(X,G)$ and $E(X,G)$ respectively.

Let $(X,G)$ be a TDS with a $G$-action, $f \in C(X,\mathbb{R})$ and $\mu \in M(X)$. We define
$$\begin{aligned}
\underline{P}_{\mu }(x,\left \{ F_{n} \right \},f,\mathbf{b} ):=\lim_{\varepsilon  \to 0} \liminf_{n \to +\infty} \frac{-\log{\mu}(B_{F_{n}}(x,\varepsilon )) +f_{F_{n}}(x)}{b(| F_{n}|) },\\
\overline{P}_{\mu }(x,\left \{ F_{n} \right \},f,\mathbf{b} ):=\lim_{\varepsilon  \to 0} \limsup_{n \to +\infty} \frac{-\log{\mu}(B_{F_{n}}(x,\varepsilon )) +f_{F_{n}}(x)}{b(| F_{n} |) }.
\end{aligned}$$
\begin{defn}\label{d1}
Let $Z\subseteq X$ be a nonempty set. The \textit{scaled measure-theoretic lower} and \textit{scaled upper local pressures} on the subset $Z$ for $\mu$ with respect to $f$, $\mathbf{b}$ and $\left\{F_{n}\right\}$ are defined by
$$
\underline{P}_{\mu }(Z,\left \{ F_{n} \right \},f ,\mathbf{b})=\int_{Z} \underline{P}_{\mu }(x,\left \{ F_{n} \right \},f,\mathbf{b} ) \ \text{d}\mu (x),\
\overline{P}_{\mu }(Z,\left \{ F_{n} \right \},f,\mathbf{b} )=\int_{Z} \overline{P}_{\mu }(x,\left \{ F_{n} \right \},f,\mathbf{b} ) \ \text{d}\mu (x).
$$
When $Z=X$, $\underline{P}_{\mu }(Z,\left \{ F_{n} \right \},f,\mathbf{b} )$ and $\overline{P}_{\mu }(Z,\left \{ F_{n} \right \},f,\mathbf{b} )$  can be simplified as $\underline{P}_{\mu }(\left \{ F_{n} \right \},f,\mathbf{b} )$ and $\overline{P}_{\mu }(\left \{ F_{n} \right \},f,\mathbf{b} )$, respectively.
\end{defn}
\begin{defn}\label{d2}
   We call the following quantities
   $$\begin{aligned}
      P_{\mu}^{B}(\left \{ F_{n} \right \}, f,\mathbf{b}):&=\lim_{\varepsilon  \to 0} \lim_{\delta  \to 0} \inf\left \{ P^{B}(\varepsilon ,Z,\left \{ F_{n} \right \}, f,\mathbf{b}):\mu(Z) \geq 1-\delta \right \}\\
      &=\lim_{\varepsilon  \to 0} \lim_{\delta  \to 0} \inf\left \{ P^{B'}(\varepsilon ,Z,\left \{ F_{n} \right \}, f,\mathbf{b}):\mu(Z) \geq 1-\delta \right \},\\
      \underline{CP}_{\mu}(\left \{ F_{n} \right \}, f,\mathbf{b}):&=\lim_{\varepsilon  \to 0} \lim_{\delta  \to 0} \inf\left \{\underline{CP}(\varepsilon ,Z,\left \{ F_{n} \right \}, f,\mathbf{b}):\mu(Z) \geq 1-\delta \right \}\\
      &=\lim_{\varepsilon  \to 0} \lim_{\delta  \to 0} \inf\left \{\underline{CP'}(\varepsilon ,Z,\left \{ F_{n} \right \}, f,\mathbf{b}):\mu(Z) \geq 1-\delta \right \},  \\
      \overline{CP}_{\mu}(\left \{ F_{n} \right \}, f,\mathbf{b}):&=\lim_{\varepsilon  \to 0} \lim_{\delta  \to 0} \inf\left \{\overline{CP}(\varepsilon ,Z,\left \{ F_{n} \right \}, f,\mathbf{b}):\mu(Z) \geq 1-\delta \right \} \\
      &=\lim_{\varepsilon  \to 0} \lim_{\delta  \to 0} \inf\left \{\overline{CP'}(\varepsilon ,Z,\left \{ F_{n} \right \}, f,\mathbf{b}):\mu(Z) \geq 1-\delta \right \},\\
      P_{\mu}^{P}(\left \{ F_{n} \right \}, f,\mathbf{b}):&=\lim_{\varepsilon  \to 0} \lim_{\delta  \to 0} \inf\left \{ P^{P}(\varepsilon ,Z,\left \{ F_{n} \right \}, f,\mathbf{b}):\mu(Z) \geq 1-\delta \right \}\\
      &=\lim_{\varepsilon  \to 0} \lim_{\delta  \to 0} \inf\left \{ P^{P'}(\varepsilon ,Z,\left \{ F_{n} \right \}, f,\mathbf{b}):\mu(Z) \geq 1-\delta \right \}
   \end{aligned}$$
   \textit{scaled Pesin-Pitskel, scaled lower capacity, scaled upper capacity}, and \textit{scaled packing topological pressures} for $\mu$ with respect to $f$, $\mathbf{b}$ and $\left\{F_{n}\right\}$.
\end{defn}
Katok \cite{Ka} introduced a type of measure-theoretic entropy. Recently, Wang \cite{Wan} studied the dimension types of measure-theoretic entropy. Zhong and Chen extended this entropy to pressure in \cite{Zh}. We shall investigate the dimension types of the scaled measure-theoretic pressure in the sense of Katok with a $G$-action.

Given $\mu \in M(X)$, $N \in \mathbb{N}$, $\alpha \in \mathbb{R}$, $\varepsilon>0$, $0<\delta<1$ and $f \in C(X,\mathbb{R})$, we define
\begin{equation}
M_{\mu}(N,\alpha ,\varepsilon ,\delta,\left\{F_{n}\right\},f,\mathbf{b})= \inf\left \{ \sum_{i} e^{-\alpha b(| F_{n_{i} }|)+f_{F_{n_{i} } } (x_{i}) } :\mu(\bigcup_{i}B_{F_{n_{i}}} (x_{i},\varepsilon )) \geq 1-\delta \right \},   \label{eq:4.1}
\end{equation}
where the infimum is taken over all finite or countable collections of $\left\{B_{F_{n_{i}}} (x_{i},\varepsilon )\right\}_{i}$ such that $x_{i} \in X$, $n_{i} \geq n$, and $\mu(\bigcup_{i}B_{F_{n_{i}}} (x_{i},\varepsilon )) \geq 1-\delta$. Likewise, we define
\begin{equation}
R_{\mu}(N,\alpha ,\varepsilon ,\delta,\left\{F_{n}\right\},f,\mathbf{b})=\inf\left \{ \sum_{i} e^{-\alpha b(| F_{N}|)+f_{F_{N} } (x_{i}) } :\mu(\bigcup_{i}B_{F_{N}} (x_{i},\varepsilon )) \geq 1-\delta \right \},   \label{eq:4.2}
\end{equation}
where the infimum is taken over all finite or countable collections of $\left\{B_{F_{N}} (x_{i},\varepsilon )\right\}_{i}$ such that $x_{i} \in X$, and $\mu(\bigcup_{i}B_{F_{N}} (x_{i},\varepsilon )) \geq 1-\delta$.

Let
$$\begin{aligned}
   M_{\mu}(\alpha ,\varepsilon ,\delta,\left\{F_{n}\right\},f,\mathbf{b})&=\lim_{N \to +\infty}M_{\mu}(N,\alpha ,\varepsilon ,\delta,\left\{F_{n}\right\},f,\mathbf{b}),\\
   \underline{R}_{\mu}(\alpha ,\varepsilon ,\delta,\left\{F_{n}\right\},f,\mathbf{b})&=\liminf_{N \to +\infty}R_{\mu}(N,\alpha ,\varepsilon ,\delta,\left\{F_{n}\right\},f,\mathbf{b}),\\
   \overline{R}_{\mu}(\alpha ,\varepsilon ,\delta,\left\{F_{n}\right\},f,\mathbf{b})&=\limsup_{N \to +\infty}R_{\mu}(N,\alpha ,\varepsilon ,\delta,\left\{F_{n}\right\},f,\mathbf{b}).
\end{aligned}$$
Define
$$M_{\mu}^{\mathcal{P}}(\alpha ,\varepsilon ,\delta,\left\{F_{n}\right\},f,\mathbf{b})=\inf\ \left \{\sum_{i=1}^{\infty}M^{P}(\alpha ,\varepsilon ,Z_{i},\left\{F_{n}\right\},f,\mathbf{b}) :\mu(\bigcup_{i=1}^{\infty}Z_{i}) \geq 1-\delta \right \}.$$
When $\alpha$ goes from $-\infty$ to $+\infty$, the quantities
$$ M_{\mu}(\alpha ,\varepsilon ,\delta,\left\{F_{n}\right\},f,\mathbf{b}),\ \underline{R}_{\mu}(\alpha ,\varepsilon ,\delta,\left\{F_{n}\right\},f,\mathbf{b}),\ \overline{R}_{\mu}(\alpha ,\varepsilon ,\delta,\left\{F_{n}\right\},f,\mathbf{b}), \text{and} \ M_{\mu}^{\mathcal{P}}(\alpha ,\varepsilon ,\delta,\left\{F_{n}\right\},f,\mathbf{b})$$
jump from $+\infty$ to $0$ at critical values respectively. Hence we can define the numbers
$$\begin{aligned}
   P_{\mu}^{KB}(\varepsilon ,\delta,\left\{F_{n}\right\},f,\mathbf{b})
   &=\sup\left\{\alpha:M_{\mu}(\alpha ,\varepsilon ,\delta,\left\{F_{n}\right\},f,\mathbf{b})=+\infty\right\}\\
   &=\inf\left\{\alpha:M_{\mu}(\alpha ,\varepsilon ,\delta,\left\{F_{n}\right\},f,\mathbf{b})=0\right\},\\
   \overline{CP}_{\mu}^{K}(\varepsilon ,\delta,\left\{F_{n}\right\},f,\mathbf{b})
   &=\sup\left\{\alpha:\overline{R}_{\mu}(\alpha ,\varepsilon ,\delta,\left\{F_{n}\right\},f,\mathbf{b})=+\infty\right\}\\
   &=\inf\left\{\alpha:\overline{R}_{\mu}(\alpha ,\varepsilon ,\delta,\left\{F_{n}\right\},f,\mathbf{b})=0\right\},\\
   \underline{CP}_{\mu}^{K}(\varepsilon ,\delta,\left\{F_{n}\right\},f,\mathbf{b})
   &=\sup\left\{\alpha:\underline{R}_{\mu}(\alpha ,\varepsilon ,\delta,\left\{F_{n}\right\},f,\mathbf{b})=+\infty\right\}\\
   &=\inf\left\{\alpha:\underline{R}_{\mu}(\alpha ,\varepsilon ,\delta,\left\{F_{n}\right\},f,\mathbf{b})=0\right\},\\
   P^{KP}_{\mu}(\varepsilon ,\delta,\left\{F_{n}\right\},f,\mathbf{b})
   &=\sup\left\{\alpha:M^{\mathcal{P} }_{\mu}(\alpha ,\varepsilon ,\delta,\left\{F_{n}\right\},f,\mathbf{b})=+\infty\right\}\\
   &=\inf\left\{\alpha:M^{\mathcal{P} }_{\mu}(\alpha ,\varepsilon ,\delta,\left\{F_{n}\right\},f,\mathbf{b})=0\right\}.
\end{aligned}$$
\begin{defn}\label{d3}
We call the following quantities
$$\begin{aligned}
   P_{\mu}^{KB}(\left\{F_{n}\right\},f,\mathbf{b})&=\lim_{\varepsilon  \to 0} \lim_{\delta \to 0} P_{\mu}^{KB}(\varepsilon ,\delta,\left\{F_{n}\right\},f,\mathbf{b}),\\
   \overline{CP}_{\mu}^{K}(\left\{F_{n}\right\},f,\mathbf{b})&=\lim_{\varepsilon  \to 0} \lim_{\delta \to 0} \overline{CP}_{\mu}^{K}(\varepsilon ,\delta,\left\{F_{n}\right\},f,\mathbf{b}),\\
   \underline{CP}_{\mu}^{K}(\left\{F_{n}\right\},f,\mathbf{b})&=\lim_{\varepsilon  \to 0} \lim_{\delta \to 0}\underline{CP}_{\mu}^{K}(\varepsilon ,\delta,\left\{F_{n}\right\},f,\mathbf{b}),\\
   P^{KP}_{\mu}(\left\{F_{n}\right\},f,\mathbf{b})&=\lim_{\varepsilon  \to 0} \lim_{\delta \to 0} P^{KP}_{\mu}(\varepsilon ,\delta,\left\{F_{n}\right\},f,\mathbf{b})
\end{aligned}$$
\textit{scaled Pesin-Pitskel, scaled lower capacity, scaled upper capacity}, and \textit{scaled packing topological pressures} for $\mu$ in the sense of Katok with respect to $f$, $\mathbf{b}$ and $\left\{F_{n}\right\}$.
\end{defn}
If we replace $f_{F_{n_{i} } } (x_{i})$ in equations \eqref{eq:4.1} and \eqref{eq:2.3} by $f_{F_{n_{i} } } (x_{i},\varepsilon )$ and $\overline{f}_{F_{n_{i} } } (x_{i},\varepsilon )$ respectively and $f_{F_{N} } (x_{i})$ in equation \eqref{eq:4.2} by $f_{F_{N} } (x_{i},\varepsilon)$, we can get new functions $\mathcal{M}_{\mu}$ , $\mathcal{R}_{\mu}$ , $\mathcal{M}^{\mathcal{P}}_{\mu}$ respectively. For any $\varepsilon>0$ and $0<\delta<1$, we denote the respective critical values by
$$P^{KB'}_{\mu}(\varepsilon,\delta,\left\{F_{n}\right\},f,\mathbf{b}), \ \overline{CP}_{\mu}^{K'}(\varepsilon ,\delta,\left\{F_{n}\right\},f,\mathbf{b}), \ \underline{CP}_{\mu}^{K'}(\varepsilon ,\delta,\left\{F_{n}\right\},f,\mathbf{b}), \ P^{KP'}_{\mu}(\varepsilon ,\delta,\left\{F_{n}\right\},f,\mathbf{b}).$$
\begin{prop}\label{p3}
   Let $(X,G)$ be a TDS with a $G$-action and a F{\o}lner  sequence $\left\{F_{n}\right\}$ satisfying Condition \ref{con1}. If $f \in C(X,\mathbb{R})$, $\mu \in M(X)$, then we have
   $$\begin{aligned}
   P^{KB}_{\mu}(\left\{F_{n}\right\},f,\mathbf{b})&=\lim_{\varepsilon \to 0}\lim_{\delta \to 0} P^{KB'}_{\mu}(\varepsilon ,\delta,\left\{F_{n}\right\},f,\mathbf{b}),\\
   \overline{CP}_{\mu}^{K}(\left\{F_{n}\right\},f,\mathbf{b})&=\lim_{\varepsilon \to 0} \lim_{\delta \to 0}\overline{CP}^{K'}_{\mu}(\varepsilon ,\delta,\left\{F_{n}\right\},f,\mathbf{b}),\\
   \underline{CP}_{\mu}^{K}(\left\{F_{n}\right\},f,\mathbf{b})&=\lim_{\varepsilon \to 0}\lim_{\delta \to 0} \underline{CP}_{\mu}^{K'}(\varepsilon ,\delta,\left\{F_{n}\right\},f,\mathbf{b}),\\
   P^{KP}_{\mu}(\left\{F_{n}\right\},f,\mathbf{b})&=\lim_{\varepsilon \to 0}\lim_{\delta \to 0} P^{KP'}_{\mu}(\varepsilon ,\delta,\left\{F_{n}\right\},f,\mathbf{b}).
   \end{aligned}$$
\end{prop}
\begin{proof}
   The proof is similar to Proposition \ref{p1}.
\end{proof}
\begin{prop}\label{p4}
   Let $(X,G)$ be a TDS with a $G$-action and a F{\o}lner sequence $\left\{F_{n}\right\}$. If $f \in C(X,\mathbb{R})$ and $\mu \in M(X)$, then we have
   $$\begin{aligned}
      P^{KB}_{\mu}(\left\{F_{n}\right\},f,\mathbf{b})=P^{B}_{\mu}(\left\{F_{n}\right\},f,\mathbf{b}), \ \underline{CP}_{\mu}^{K}(\left\{F_{n}\right\},f,\mathbf{b})=\underline{CP}_{\mu}(\left\{F_{n}\right\},f,\mathbf{b}),\\
      \overline{CP}_{\mu}^{K}(\left\{F_{n}\right\},f,\mathbf{b}) \leq \overline{CP}_{\mu}(\left\{F_{n}\right\},f,\mathbf{b}), \  P^{KP}_{\mu}(\left\{F_{n}\right\},f,\mathbf{b})= P^{P}_{\mu}(\left\{F_{n}\right\},f,\mathbf{b}).
   \end{aligned}$$
\end{prop}
\begin{proof}
We first show that $P^{KB}_{\mu}(\left\{F_{n}\right\},f,\mathbf{b})=P^{B}_{\mu}(\left\{F_{n}\right\},f,\mathbf{b})$. For any $N \in \mathbb{N}$, $\alpha \in \mathbb{R}$, $\varepsilon>0$, $0<\delta<1$, and $Z \subseteq X$ with $\mu(Z) \geq 1-\delta$, we have
$$M_{\mu}(N,\alpha ,\varepsilon ,\delta,\left\{F_{n}\right\},f,\mathbf{b}) \leq M(N,\alpha ,\varepsilon ,Z,\left\{F_{n}\right\},f,\mathbf{b}).$$
Letting $N \to +\infty$, we have
$$M_{\mu}(\alpha ,\varepsilon ,\delta,\left\{F_{n}\right\},f,\mathbf{b}) \leq M(\alpha ,\varepsilon ,Z,\left\{F_{n}\right\},f,\mathbf{b}).$$
This yields
$$P_{\mu}^{KB}(\varepsilon ,\delta,\left\{F_{n}\right\},f,\mathbf{b}) \leq P^{B}(\varepsilon ,Z,\left\{F_{n}\right\},f,\mathbf{b}),$$
and then
$$P_{\mu}^{KB}(\varepsilon ,\delta,\left\{F_{n}\right\},f,\mathbf{b}) \leq \text{inf} \left\{ P^{B}(\varepsilon ,Z,\left\{F_{n}\right\},f,\mathbf{b}):\mu(Z) \geq 1-\delta \right\}.$$
Letting $\delta \to 0$ and $\varepsilon \to 0$, we have
$$P^{KB}_{\mu}(\left\{F_{n}\right\},f,\mathbf{b}) \leq P^{B}_{\mu}(\left\{F_{n}\right\},f,\mathbf{b}).$$
The verifications of $\underline{CP}_{\mu}^{K}(\left\{F_{n}\right\},f,\mathbf{b}) \leq \underline{CP}_{\mu}(\left\{F_{n}\right\},f,\mathbf{b})$ and $\overline{CP}_{\mu}^{K}(\left\{F_{n}\right\},f,\mathbf{b}) \leq \overline{CP}_{\mu}(\left\{F_{n}\right\},f,\mathbf{b})$ are similarly.

Now we prove $P^{KB}_{\mu}(\left\{F_{n}\right\},f,\mathbf{b}) \geq P^{B}_{\mu}(\left\{F_{n}\right\},f,\mathbf{b})$. Let $P^{KB}_{\mu}(\left\{F_{n}\right\},f,\mathbf{b})=a$. Then for any $s>0$, there exists $\varepsilon'>0$ such that
$$\lim_{\delta \to 0} P_{\mu}^{KB}(\varepsilon ,\delta,\left\{F_{n}\right\},f,\mathbf{b})<a+s, \  \forall 0<\varepsilon<\varepsilon'.$$
Hence for any $\varepsilon \in (0,\varepsilon')$, there exists $\delta_{\varepsilon}$ such that
$$P_{\mu}^{KB}(\varepsilon ,\delta,\left\{F_{n}\right\},f,\mathbf{b})<a+s, \ \forall 0<\delta<\delta_{\varepsilon}.$$
It implies that $M_{\mu}(a+s ,\varepsilon ,\delta,\left\{F_{n}\right\},f,\mathbf{b})=0$. For any $k \in \mathbb{N}$, we can find a sequence $\{ \delta_{k,m} \}_{m \in \mathbb{N}}$ with $\lim\limits_{m \to +\infty}\delta_{k,m}=0$ and $\left\{B_{F_{n_{i}}} (x_{i},\varepsilon )\right\}_{i \in I_{k,m}}$ such that $x_{i} \in X$, $n_{i} \geq k$ for all $i$, $$\mu(\bigcup\limits_{i \in I_{k,m}}B_{F_{n_{i}}} (x_{i},\varepsilon )) \geq 1-\delta_{k,m}$$ and
$$\sum_{i \in I_{k,m}} e^{-(a+s) b(| F_{n_{i} }|)+f_{F_{n_{i} } } (x_{i}) } \leq \frac{1}{2^{m}}.$$
Let $$Z_{k}=\bigcup_{m \in \mathbb{N}}\bigcup_{i \in I_{k,m}}B_{F_{n_{i}}} (x_{i},\varepsilon ).$$
Then $\mu(Z_{k})=1$ and
$$\begin{aligned}
M(k,a+s ,\varepsilon ,Z_{k},\left\{F_{n}\right\},f,\mathbf{b}) &\leq \sum_{m \in \mathbb{N}}\sum_{i \in I_{k,m}} e^{-(a+s) b(| F_{n_{i} } |)+f_{F_{n_{i} } } (x_{i}) }\\
&\leq \sum_{m \in \mathbb{N}}\frac{1}{2^{m}} \leq 1.
\end{aligned}$$
Let $Z_{\varepsilon} = \bigcap_{k \in \mathbb{N}}Z_{k}$. Then $\mu(Z_{\varepsilon})=1$ and
$$M(k,a+s ,\varepsilon ,Z_{\varepsilon},\left\{F_{n}\right\},f,\mathbf{b}) \leq M(k,a+s ,\varepsilon ,Z_{k},\left\{F_{n}\right\},f,\mathbf{b}) \leq 1, \ \forall k \in \mathbb{N}.$$
This implies that $P^{B}(\varepsilon ,Z_{\varepsilon},\left\{F_{n}\right\},f,\mathbf{b}) \leq a+s$. Hence
$$ P_{\mu}^{B}(\left \{ F_{n} \right \}, f,\mathbf{b})=\lim_{\varepsilon  \to 0} \lim_{\delta  \to 0} \inf\left \{ P^{B}(\varepsilon ,Z,\left \{ F_{n} \right \}, f,\mathbf{b}):\mu(Z) \geq 1-\delta \right \} \leq a+s.$$
Letting $s \to 0$, we have $P_{\mu}^{B} (\left\{F_{n}\right\},f,\mathbf{b})\leq P^{KB}_{\mu}(\left\{F_{n}\right\},f,\mathbf{b}).$

We shall prove $\underline{CP}_{\mu}^{K}(\left\{F_{n}\right\},f,\mathbf{b}) \geq \underline{CP}_{\mu}(\left\{F_{n}\right\},f,\mathbf{b})$. Let $\underline{CP}_{\mu}^{K}(\left\{F_{n}\right\},f,\mathbf{b})=a$. Then for any $s>0$, there exists $\varepsilon'>0$ such that for any $\varepsilon \in (0,\varepsilon')$, there exists $\delta_{\varepsilon}$ with
$$\liminf_{N \to +\infty}R_{\mu}(N,a+s ,\varepsilon ,\delta,\left\{F_{n}\right\},f,\mathbf{b})=0, \ \forall \delta \in (0,\delta_{\varepsilon}).$$
Fix $\delta \in (0,\delta_{\varepsilon})$. For any $m \in \mathbb{N}$, we have
$$\liminf_{N \to +\infty}R_{\mu}(N,a+s ,\varepsilon ,\frac{\delta}{2^{m}},\left\{F_{n}\right\},f,\mathbf{b})=0.$$
Then for any $m \in \mathbb{N}$, there exists a family $\left\{B_{F_{k_{m}}} (x_{i},\varepsilon )\right\}_{i \in I_{m}}$ such that $$\mu(\bigcup\limits_{i \in I_{m}}B_{F_{k_{m}}} (x_{i},\varepsilon )) \geq 1-\frac{\delta}{2^{m}}$$ and
$$\sum_{i \in I_{m}} e^{-(a+s) b(| F_{k_{m} }|)+f_{F_{k_{m} } } (x_{i}) } \leq 1.$$
Let $Z_{\delta}=\bigcap_{m \in \mathbb{N}}(\bigcup\limits_{i \in I_{m}}B_{F_{k_{m}}} (x_{i},\varepsilon ))$. Then $\mu(Z_{\delta}) \geq 1-\delta$ and
$$\liminf_{N \to +\infty}R(N,a+s ,\varepsilon ,Z_{\delta},\left\{F_{n}\right\},f,\mathbf{b}) \leq 1.$$
This implies that $\underline{CP}(\varepsilon,Z_{\delta},\left\{F_{n}\right\},f,\mathbf{b}) \leq a+s$. Hence
$$\underline{CP}_{\mu}(\left\{F_{n}\right\},f,\mathbf{b})=\lim_{\varepsilon  \to 0} \lim_{\delta  \to 0} \inf\left \{\underline{CP}(\varepsilon ,Z,\left \{ F_{n} \right \}, f,\mathbf{b}):\mu(Z) \geq 1-\delta \right \} \leq a+s.$$
Letting $s \to 0$, the desired inequality follows.

We now prove that $P^{KP}_{\mu}(\left\{F_{n}\right\},f,\mathbf{b}) \leq P^{P}_{\mu}(\left\{F_{n}\right\},f,\mathbf{b})$. Let $P^{KP}_{\mu}(\left\{F_{n}\right\},f,\mathbf{b})>s$. Then there exist $\varepsilon'>0$ and $\delta'>0$ such that
$$P^{KP}_{\mu}(\varepsilon,\delta,\left\{F_{n}\right\},f,\mathbf{b})>s, \ \forall \varepsilon \in (0,\varepsilon')\ \text{and} \ \delta \in (0,\delta').$$
This implies that $M_{\mu}^{\mathcal{P}}(s,\varepsilon ,\delta,\left\{F_{n}\right\},f,\mathbf{b})=+\infty$. For any $Z \subseteq X$ with $\mu(Z) \geq 1-\delta$, if $Z \subseteq \bigcup_{i}Z_{i}$, then $\mu(\bigcup_{i}Z_{i}) \geq 1-\delta$. Thus
$$\sum_{i=1}^{\infty}M^{P}(s,\varepsilon ,Z_{i},\left\{F_{n}\right\},f,\mathbf{b}) \geq M^{P}(s,\varepsilon ,Z,\left\{F_{n}\right\},f,\mathbf{b})=+\infty.$$
Hence $M^{\mathcal{P}}(s,\varepsilon ,Z,\left\{F_{n}\right\},f,\mathbf{b})=+\infty$. It follows that
$$P_{\mu}^{P}(\left \{ F_{n} \right \}, f,\mathbf{b})=\lim_{\varepsilon  \to 0} \lim_{\delta  \to 0} \text{inf}\left \{ P^{P}(\varepsilon ,Z,\left \{ F_{n} \right \}, f,\mathbf{b}):\mu(Z) \geq 1-\delta \right \} \geq s.$$
Then we have $P^{KP}_{\mu}(\left\{F_{n}\right\},f,\mathbf{b}) \leq P^{P}_{\mu}(\left\{F_{n}\right\},f,\mathbf{b})$.

We shall show the inverse inequality. If $ P^{P}_{\mu}(\left\{F_{n}\right\},f,\mathbf{b})>s$, then there exist $\varepsilon'>0$ and $\delta'>0$ such that
$$\text{inf}\left \{ P^{P}(\varepsilon ,Z,\left \{ F_{n} \right \}, f,\mathbf{b}):\mu(Z) \geq 1-\delta \right \} > s,\forall \varepsilon \in (0,\varepsilon') \ \text{and} \ \delta \in (0,\delta').$$
For any $\left\{Z_{i} \right\}_{i \geq 1}$ with $\mu(\bigcup_{i}Z_{i}) \geq 1-\delta$, we have
$$P^{P}(\varepsilon ,\bigcup_{i}Z_{i},\left \{ F_{n} \right \}, f,\mathbf{b})>s.$$
Then
$$M^{\mathcal{P}}(s,\varepsilon ,\bigcup_{i}Z_{i},\left \{ F_{n} \right \}, f,\mathbf{b})=+\infty.$$
Thus
$$\sum_{i=1}^{\infty}M^{P}(s,\varepsilon ,Z_{i},\left\{F_{n}\right\},f,\mathbf{b})=+\infty.$$
Hence
$$M_{\mu}^{\mathcal{P}}(s,\varepsilon ,\delta,\left \{ F_{n} \right \}, f,\mathbf{b})=+\infty.$$
This implies that
$$P^{KP}_{\mu}(\varepsilon,\delta,\left\{F_{n}\right\},f,\mathbf{b}) \geq s.$$
Since $\varepsilon$ and $\delta$ are arbitrary, we have $P^{KP}_{\mu}(\left\{F_{n}\right\},f,\mathbf{b}) \geq P^{P}_{\mu}(\left\{F_{n}\right\},f,\mathbf{b})$.
\end{proof}
 With the help of Proposition \ref{p4}, we can get the relationship among the above scaled measure-theoretic pressures.
\begin{thm} \label{T5}
   Let $(X,G)$ be a TDS with a $G$-action and a F{\o}lner  sequence $\left\{F_{n}\right\}$ satisfying Condition \ref{con1}. If $f \in C(X,\mathbb{R})$ and $\mu \in M(X)$, then we have
   $$\begin{aligned}
   P^{KB}_{\mu}(\left\{F_{n}\right\},f,\mathbf{b})=P^{B}_{\mu}(\left\{F_{n}\right\},f,\mathbf{b}) \leq P^{KP}_{\mu}(\left\{F_{n}\right\},f,\mathbf{b})&=P^{P}_{\mu}(\left\{F_{n}\right\},f,\mathbf{b}) \\
   &\leq \overline{CP}_{\mu}^{K}(\left\{F_{n}\right\},f,\mathbf{b}) \leq \overline{CP}_{\mu}(\left\{F_{n}\right\},f,\mathbf{b}).
   \end{aligned}$$
\end{thm}
\begin{proof}
   We now prove $P^{KP}_{\mu}(\left\{F_{n}\right\},f,\mathbf{b})\leq \overline{CP}_{\mu}^{K}(\left\{F_{n}\right\},f,\mathbf{b})$. The desired result then follows from Proposition \ref{p4}.

   Let $P^{KP}_{\mu}(\left\{F_{n}\right\},f,\mathbf{b})>s>t>-\infty$. Then there exist $\varepsilon'>0$ and $\delta'>0$ such that
$$P^{KP}_{\mu}(\varepsilon,\delta,\left\{F_{n}\right\},f,\mathbf{b})>s, \ \forall \varepsilon \in (0,\varepsilon')\ \text{and} \ \delta \in (0,\delta').$$
This implies that $M_{\mu}^{\mathcal{P}}(s,\varepsilon ,\delta,\left\{F_{n}\right\},f,\mathbf{b})=+\infty.$ For any $Z \subseteq X$ with $\mu(Z) \geq 1-\delta,$ if $Z \subseteq \bigcup_{i}Z_{i}$, then $\mu(\bigcup_{i}Z_{i}) \geq 1-\delta.$ Hence $M^{P}(s,\varepsilon ,Z,\left\{F_{n}\right\},f,\mathbf{b})=+\infty$.
Since $\left\{F_{n}\right\}$ and $\mathbf{b}$ satisfy Condition \ref{con1}, we assume that $$\sum_{k \geq 1}e^{b(|F_{k}|) (t-s)}=M.$$ Hence for any $ N \in \mathbb{N}$, there exists a countable pairwise disjoint family $\left\{\overline{B}_{F_{n_{i}}} (x_{i},\varepsilon )\right\}_{i}$ such that $x_{i} \in Z$, $n_{i} \geq N$ for all $i$ and $ \sum_{i} e^{-sb(| F_{n_{i} }|)+f_{F_{n_{i} } } (x_{i}) } >M$. For each $k$, let
$$m_{k}=\left \{ x _{i}:n_{i}=k\right \} .$$
Then
$$\sum_{k=N}^{\infty}\sum _{x \in m_{k}}e^{f_{F_{k}}(x)}e^{-b(| F_{k} |)s }>M.$$
Thus there exists $k \geq N$ such that
$$\sum_{x \in m_{k}}e^{f_{F_{k}}(x)}e^{-b(| F_{k}|) t }  \geq 1-e^{t-s}.$$
Taking a collection $\left \{ B_{F_{k} } (y_{i},\frac{\varepsilon }{2} )\right \} _{i \in I}$ such that $Z \subseteq \bigcup_{i \in I }B_{F_{k}}(y_{i},\frac{\varepsilon}{2})$. Then $\mu(\bigcup_{i \in I }B_{F_{k}}(y_{i},\frac{\varepsilon}{2})) \geq 1-\delta$. We can obtain that for any $x_{1},x_{2} \in m_{k}$, there exist $y_{1}$, $ y_{2}$ with $y_{1} \neq y_{2}$ such that $x_{i} \in B_{F_{k}}(y_{i},\frac{\varepsilon}{2})$, $i=1,2$. Then
$$\mathcal{R}_{\mu}(k,t,\frac{\varepsilon}{2},\delta,\left\{F_{n}\right\},f,\mathbf{b}) \geq \sum_{x \in m_{k}}e^{f_{F_{k}}(x)}e^{-b(|F_{k}|)t} \geq 1-e^{t-s}>0.$$
Thus $\overline{CP}_{\mu}^{K'}(\varepsilon/2,\delta,\left\{F_{n}\right\},f,\mathbf{b}) \geq t$. Letting $\varepsilon \to 0$ and $\delta \to 0$, we can get that $\overline{CP}_{\mu}^{K}(\left\{F_{n}\right\},f,\mathbf{b}) \geq t$. Hence $P^{KP}_{\mu}(\left\{F_{n}\right\},f,\mathbf{b})\leq \overline{CP}_{\mu}^{K}(\left\{F_{n}\right\},f,\mathbf{b})$.
\end{proof}
Ma and Wen \cite{Ma} showed that an analogue of Billingsley's Theorem for the Hausdorff dimension. Tang et al. \cite{Ta} extended this result to the Pesin-Pitskel topological pressure. Zhong and Chen \cite{Zh} proved Billingsley's Theorem for the packing pressure when $G=\mathbb{Z}$.
We shall generalize Billingsley's Theorem for the scaled packing pressure with a $G$-action, which shows that the scaled packing pressure can be determined by the scaled measure-theoretic upper local pressure.
\begin{thm}\label{T2}
   Let $(X,G)$ be a TDS with a $G$-action and a F{\o}lner  sequence $\left\{F_{n}\right\}$ satisfying Condition \ref{con1}. If $f \in C(X,\mathbb{R})$, $\mu \in M(X)$ and $Z \subset X$. Then for $s \in \mathbb{R}$, the following properties hold:
   \begin{enumerate}[(1)]
      \item if $\overline{P}_{\mu }(\left \{ F_{n} \right \},f ,\mathbf{b}) \leq s$ for all $x \in Z$, then $P^{P}(Z,\left\{F_{n}\right\},f,\mathbf{b}) \leq s$;
      \item if $\overline{P}_{\mu }(\left \{ F_{n} \right \},f ,\mathbf{b}) \geq s$ for all $x \in Z$ and $\mu(Z)>0$, then $ P^{P}(Z,\left\{F_{n}\right\},f,\mathbf{b}) \geq s$.
   \end{enumerate}
   \end{thm}
To prove Theorem \ref{T2}, we need the following lemma (cf. \cite[Theorem 2.1]{Mat}).
\begin{lem}\label{l2}
   (5r-lemma) Let $(X,d)$ be a compact metric space and $\mathcal{B}=\left \{ B(x_{i},r_{i}) \right \} _{i \in I}$ be a family of open (or closed) balls in $X$. Then there exists a finite or countable subfamily $\mathcal{B'}=\left \{ B(x_{i},r_{i}) \right \} _{i \in I'}$ of pairwise disjoint balls in $\mathcal{B}$ such that
   $$\bigcup_{B \in \mathcal{B} }B \subset \bigcup_{i \in I'}B(x_{i},5r_{i}) .$$
\end{lem}
\textit{Proof of Theorem 3.7}.
We now prove the first assertion. For a fixed $\beta>s$, let
$$Z_{m}=\left \{ x \in Z:\limsup_{n \to +\infty} \frac{-\log{\mu}(B_{F_{n}}(x,\varepsilon ))+f_{F_{n}} (x)}{b(|  F_{n} |)} \leq\frac{\beta+s}{2} ,\ \forall \varepsilon  \in (0,\frac{1}{m} ) \right \}.$$
It is easy to see that $Z=\bigcup_{m=1}^{\infty}Z_{m}$, since $\overline{P}_{\mu }(\left \{ F_{n} \right \},f ,\mathbf{b}) \leq s$ for all $x \in Z$. Fix $m\geq1$ and $\varepsilon \in (0,\frac{1}{m})$. For each $x \in Z_{m}$, there exists $N \in \mathbb{N}$ such that
$$\mu(B_{F_{n}}(x,\varepsilon ))\geq e^{-b(|F_{n}|)(\frac{\beta +s}{2})+f_{F_{n}}(x)}, \ \forall n\geq N.$$
Let
$$Z_{m,N}=\left \{ x \in Z_{m}:  \mu(B_{F_{n}}(x,\varepsilon ))\geq e^{-b(|F_{n}|)(\frac{\beta +s}{2})+f_{F_{n}}(x)}, \ \forall n\geq N\right \}.$$
For the above $N$, we can choose $L$ large enough such that $L>>N$. For a finite or countable disjoint family $\mathcal{F}=\left\{\overline{B}_{F_{n_{i}}}(x_{i},\varepsilon) \right\}_{i}$ with $x_{i} \in Z_{m,N}$ and $n_{i}\geq L$, we have $b(|F_{n_{i}}|)\geq \log{L}$, since $\mathbf{b}$ and $\left\{F_{n}\right\}$ satisfying Condition \ref{con1}. Hence
$$\begin{aligned}
   \sum_{i}e^{-b(|F_{n_{i}}|)\beta + f_{F_{n_{i}}}(x_{i})}&\leq L^{-(\beta -s)/2}\sum_{i}e^{-b(|F_{n_{i}}|)[(\beta +s)/2]+f_{F_{n_{i}}}(x_{i})} \\
   &\leq L^{-(\beta -s)/2}\sum_{i}\mu(\overline{B}_{F_{n_{i}}}(x_{i},\varepsilon)) \\
   &\leq  L^{-(\beta -s)/2}.
\end{aligned}$$
Since $\mathcal{F}$ is arbitrary, it follows that
$$M^{P}(L,\beta,\varepsilon,Z_{m,N},\left\{F_{n}\right\},f,\mathbf{b})\leq L^{-(\beta -s)/2}.$$
Hence $M^{P}(\beta,\varepsilon,Z_{m,N},\left\{F_{n}\right\},f,\mathbf{b})=0$. It is easy to see that $Z_{m}=\bigcup_{N=1}^{\infty}Z_{m,N}$. Then we have $M^{\mathcal{P}}(\beta,\varepsilon,Z_{m},\left\{F_{n}\right\},f,\mathbf{b})=0$. This implies that $P^{P}(Z_{m},\left\{F_{n}\right\},f,\mathbf{b}) \leq \beta$. Thus
$$P^{P}(Z,\left\{F_{n}\right\},f,\mathbf{b})=\sup_{m}P^{P}(Z_{m},\left\{F_{n}\right\},f,\mathbf{b}) \leq \beta.$$
Since $\beta$ is arbitrary, then we have $P^{P}(Z,\left\{F_{n}\right\},f,\mathbf{b}) \leq s$.

We shall prove the second assertion. Fix $\beta<s $ and let $\delta=\frac{s-\beta}{2}$. Since $\overline{P}_{\mu }(\left \{ F_{n} \right \},f,\mathbf{b} ) \geq s$, there exists $\varepsilon'>0$ such that
   $$\limsup_{n \to +\infty} \frac{-\log{\mu}(B_{F_{n}}(x,\varepsilon )) +f_{F_{n}}(x)}{b(| F_{n}|) }>\beta +\delta , \ \forall \varepsilon \in (0,\varepsilon').$$
   For any $E \subseteq Z$ with $\mu(E)>0$, we define
   $$E_{n}=\left \{ x \in E:\mu(B_{F_{n}}(x,\varepsilon ))<e^{-b(| F_{n}|)(\beta+\delta)+f_{F_{n}}(x) } \right \}, \ n \in \mathbb{N}.$$
   Then we have $E=\bigcup_{n=N}^{\infty}E_{n}$ for any $N \in \mathbb{N}$. Fix $N \in \mathbb{N}$, there exists $n \geq N$ such that
   $$\mu(E_{n}) \geq \frac{1}{n(n+1)}\mu(E).$$
   Fix such $n$ and let $\mathcal{B}=\left \{ B_{F_{n}}(x,\frac{\varepsilon }{5} ):x \in E_{n} \right \} $. By Lemma \ref{l2} (using the metric $d_{F_{n}}$ instead of $d$), there exists a finite pairwise disjoint family $\left\{ B_{F_{n}}(x_{i},\frac{\varepsilon }{5}) \right\}_{i \in I}$ with $x_{i} \in E_{n}$ such that
   $$E_{n} \subset \bigcup_{x \in E_{n}}B_{F_{n}}(x,\frac{\varepsilon}{5}) \subset \bigcup_{i \in I} B_{F_{n}}(x_{i},\varepsilon).$$
   Hence
   $$\begin{aligned}
      M^{P}(N,\beta,\frac{\varepsilon}{5},E,\left\{F_{n}\right\},f,\mathbf{b}) &\geq M^{P}(N,\beta,\frac{\varepsilon}{5},E_{n},\left\{F_{n}\right\},f,\mathbf{b}) \geq \sum_{i \in I}e^{-b( | F_{n} |) \beta+f_{F_{n}}(x_{i})}\\
      &=e^{b(|F_{n}|)\delta}\sum_{i \in I}e^{-b(| F_{n} |) (\beta+\delta)+f_{F_{n}}(x_{i})} \geq e^{b(|F_{n}|)\delta}\sum_{i \in I}\mu(B_{F_{n}}(x_{i},\varepsilon)) \\
      &\geq e^{b(|F_{n}|)\delta}\mu(E_{n}) \geq e^{b(|F_{n}|)\delta} \frac{\mu(E)}{n(n+1)}.
   \end{aligned}$$
   Since $\mathbf{b}$ and $\left\{F_{n}\right\}$ satisfying Condition \ref{con1}, we have
   $$\lim_{n \to +\infty}\frac{e^{b(|F_{n}|)\delta}}{n(n+1)}=+\infty.$$
   Hence
   $$ M^{P}(\beta,\frac{\varepsilon}{5},E,\left\{F_{n}\right\},f,\mathbf{b})=+\infty.$$
   Since $E \subseteq Z$ and $\beta<s$ are arbitrary, we have $M^{P}(s,\frac{\varepsilon}{5},Z,\left\{F_{n}\right\},f,\mathbf{b})=+\infty$. It implies that
   $$ P^{P}(Z,\left\{F_{n}\right\},f,\mathbf{b}) \geq  P^{P}(\frac{\varepsilon}{5},Z,\left\{F_{n}\right\},f,\mathbf{b}) \geq s.$$
\begin{rem}
   Zhong and Chen \cite{Zh} proved Theorem \ref{T2} when $G=\mathbb{Z}$ and ${b(|F_{n}|)=|F_{n}|}$ for all $n\in \mathbb{N}$.
\end{rem}

\begin{prop}\label{p5}
   Let $(X,G)$ be a TDS with a $G$-action and a F{\o}lner  sequence $\left\{F_{n}\right\}$ satisfying Condition \ref{con1}. If $f \in C(X,\mathbb{R})$ and $\mu \in M(X)$, then
   $$ \overline{P}_{\mu}(\left\{F_{n}\right\},f,\mathbf{b}) \leq P_{\mu}^{KP}(\left\{F_{n}\right\},f,\mathbf{b}).$$
\end{prop}
\begin{proof}
   For any $s<\overline{P}_{\mu}(\left\{F_{n}\right\},f,\mathbf{b})$, we can find $\varepsilon,\ \beta>0$ and $A \subseteq X$ with $\mu(A)>0$ such that
   $$\limsup_{n \to +\infty} \frac{-\log{\mu}(B_{F_{n}}(x,\varepsilon )) +f_{F_{n}}(x)}{b(| F_{n}|) }>\beta +s , \ \forall x \in A.$$
   Fix $\delta \in (0,\mu(A))$, we shall show that
   $$M_{\mu}^{\mathcal{P}}(s,\frac{\varepsilon}{5},\delta,\left \{ F_{n} \right \}, f,\mathbf{b})=+\infty.$$
   Let $\left\{Z_{i}\right\}_{i \in I}$ be a countable family with $\mu(\bigcup_{i}Z_{i})>1-\delta$. It follows that
   $$\mu(A \cap \bigcup_{i}Z_{i}) \geq \mu(A)-\delta >0.$$
   Hence there exists $i$ such that $\mu(A \cap Z_{i})>0$. For such $i$, we define
   $$E_{n}=\left \{ x \in A \cap Z_{i}:\mu(B_{F_{n}}(x,\varepsilon ))<e^{-b(| F_{n} |)(\beta+s)+f_{F_{n}}(x) } \right \}, \ n \in \mathbb{N}.$$
   It is clear that $A \cap Z_{i}=\bigcup^{\infty}_{n=N}E_{n}$ for each $N \in \mathbb{N}$. Fix $N \in \mathbb{N}$, then $\sum_{n=N}^{\infty }\mu(E_{n}) \geq \mu(A \cap Z_{i}) $. It follows that there exists $n \geq N$ such that
   $$\mu(E_{n})\geq \frac{1}{n(n+1)}\mu(A \cap Z_{i}).$$
   Fix such $n$ and let $\mathcal{B}=\left \{ B_{F_{n}}(x,\frac{\varepsilon }{5} ):x \in E_{n} \right \} $. Then there exists a finite pairwise disjoint family $\left\{ B_{F_{n}}(x_{i},\frac{\varepsilon }{5}) \right\}_{i \in I}$ with $x_{i} \in E_{n}$ such that
   $$E_{n} \subset \bigcup_{x \in E_{n}}B_{F_{n}}(x,\frac{\varepsilon}{5}) \subset \bigcup_{i \in I} B_{F_{n}}(x_{i},\varepsilon).$$
   Hence
   $$\begin{aligned}
      M^{P}(N,s,\frac{\varepsilon}{5},Z_{i},\left\{F_{n}\right\},f,\mathbf{b}) &\geq M^{P}(N,s,\frac{\varepsilon}{5},A\cap Z_{i},\left\{F_{n}\right\},f,\mathbf{b}) \geq M^{P}(N,s,\frac{\varepsilon}{5},E_{n},\left\{F_{n}\right\},f,\mathbf{b})\ \\
      &\geq \sum_{i \in I}e^{-b(| F_{n}|) s+f_{F_{n}}(x_{i})} =e^{b(|F_{n}|)\beta}\sum_{i \in I}e^{-b(| F_{n} |) (\beta+s)+f_{F_{n}}(x_{i})} \\
      &\geq e^{b(|F_{n}|)\beta}\sum_{i \in I}\mu(B_{F_{n}}(x_{i},\varepsilon)) \geq e^{b(|F_{n}|)\beta}\mu(E_{n}) \geq e^{b(|F_{n}|)\beta} \frac{\mu(E)}{n(n+1)}.
   \end{aligned}$$
   Since $\mathbf{b}$ and $\left\{F_{n}\right\}$ satisfy
   $$\lim_{n \to +\infty}\frac{e^{b(|F_{n}|)\beta}}{n(n+1)}=+\infty,$$
   we have
   $$M^{P}(s,\frac{\varepsilon}{5},Z_{i},\left\{F_{n}\right\},f,\mathbf{b})=\lim_{N \to +\infty}M^{P}(N,s,\frac{\varepsilon}{5},Z_{i},\left\{F_{n}\right\},f,\mathbf{b})=+\infty.$$
   It follows that $\overline{P}_{\mu}(\left\{F_{n}\right\},f,\mathbf{b}) \leq P_{\mu}^{KP}(\left\{F_{n}\right\},f,\mathbf{b})$.
\end{proof}

\section{Variational principle}\label{sec4}
In this section, we will prove the variational principle among the scaled packing topological pressure and other kinds of pressures.
\begin{thm}\label{T3}
    Let $(X,G)$ be a TDS with a $G$-action and a F{\o}lner  sequence $\left\{F_{n}\right\}$ satisfying Condition \ref{con1}. Suppose that $f \in C(X,\mathbb{R})$, $Z \subset X$ is a nonempty compace set. If  $ P^{P}(Z,\left\{F_{n}\right\},f) > \|f\|$, where $\|f\|:=\sup_{x \in X}|f(x)|$, then
   $$\begin{aligned}
      P^{P}(Z,\left\{F_{n}\right\},f,\mathbf{b})&=\sup \{ \overline{P}_{\mu }(\left \{ F_{n} \right \},f,\mathbf{b} ):\mu \in M(X), \ \mu(Z)=1 \}\\
      &=\sup \{P_{\mu }^{P}(\left \{ F_{n} \right \},f ,\mathbf{b}):\mu \in M(X), \ \mu(Z)=1 \}\\
      &=\sup \{P_{\mu }^{KP}(\left \{ F_{n} \right \},f ,\mathbf{b}):\mu \in M(X), \ \mu(Z)=1 \}.\\
   \end{aligned}$$
\end{thm}
To prove Theorem \ref{T3}, we need the following lemma.
\begin{lem}\label{l3}
   Let $Z \subset X$, $\varepsilon>0$ and $s>\|f\|$. If $\mathbf{b}$ and $\left\{F_{n}\right\}$ satisfy Condition \ref{con1} and $M^{P}(s,\varepsilon,Z,\{F_{n}\},f,\mathbf{b})=\infty$, then for a given finite interval $(a,b)\in [0,+\infty)$ and $N \in \mathbb{N}$, there exists a finite disjoint collection $\{\overline{B}_{F_{n_{i}}}(x_{i},\varepsilon)\}$ such that $x_{i}\in Z$, $n_{i} \geq N$ and $\sum_{i}e^{-sb(|F_{n_{i}}|)+f_{F_{n_{i}}}(x_{i})}\in (a,b).$
\end{lem}
\begin{proof}
   Let $N_{1}>N$ be large enough such that $e^{b(| F_{N_{1}}|) (\|f\|-s)}<b-a$. Since $M^{P}(s,\varepsilon,Z,\{F_{n}\},f,\mathbf{b})=+\infty$, we have  $M^{P}(N_{1},s,\varepsilon,Z,\{F_{n}\},f,\mathbf{b})=+\infty$. Then there exists a finite disjoint collection $\{\overline{B}_{F_{n_{i}}}(x_{i},\varepsilon)\}$ such that $x_{i} \in Z$, $n_{i} \geq N_{1}$ and $\sum_{i}e^{-sb(|F_{n_{i}}|)+f_{F_{n_{i}}}(x_{i})}>b$. Since
   $$e^{-sb(|F_{n_{i}}|)+f_{F_{n_{i}}}(x_{i})}\leq e^{-sb(| F_{n_{i}}|)+b(| F_{n_{i}}|)\|f\|} \leq e^{b(| F_{n_{i}}|) (\|f\|-s)}<b-a,$$
   we can discard elements in this collection one by one until $\sum_{i}e^{-sb(|F_{n_{i}}|)+f_{F_{n_{i}}}(x_{i})} \in (a,b)$.
\end{proof}
We now prove Theorem \ref{T3} by using the approach of Feng and Huang in \cite{Fe}.

\textit{Proof of Theorem 4.1}.
   Using Proposition \ref{p4} and Proposition \ref{p5}, we have
   $$\begin{aligned}
      P^{P}(Z,\left\{F_{n}\right\},f,\mathbf{b})& \geq \sup \{P_{\mu }^{P}(\left \{ F_{n} \right \},f,\mathbf{b} ):\mu \in M(X), \ \mu(Z)=1 \}\\
      &\geq\sup \{P_{\mu }^{KP}(\left \{ F_{n} \right \},f ,\mathbf{b}):\mu \in M(X), \ \mu(Z)=1 \}\\
      &\geq\sup \{ \overline{P}_{\mu }(\left \{ F_{n} \right \},f ,\mathbf{b}):\mu \in M(X), \ \mu(Z)=1 \}.\\
   \end{aligned}$$
   We shall prove the inequality
   $$P^{P}(Z,\left\{F_{n}\right\},f,\mathbf{b})\leq \sup \{ \overline{P}_{\mu }(\left \{ F_{n} \right \},f ,\mathbf{b}):\mu \in M(X), \ \mu(Z)=1 \}.$$
   For any $s \in (\|f\|,P^{P}(Z,\{F_{n}\},f,\mathbf{b}))$, we take $\varepsilon$ small enough such that $s<P^{P}(\varepsilon,Z,\{F_{n}\},f,\mathbf{b})$. Fix $t \in (s,P^{P}(\varepsilon,Z,\{F_{n}\},f,\mathbf{b}))$. We will construct inductively the following four sequences:
   \begin{enumerate}[1)]
      \item[1)] a sequence of finite sets $\left\{K_{i} \right\}$ with $K_{i}\subset Z$;
      \item[2)] a sequence of finite measures $\left\{\mu_{i}\right\}$ with each $\mu_{i}$ being supported on $K_{i}$;
      \item[3)] a sequence of positive numbers $\left\{\gamma_{i}\right\}$;
      \item[4)] a sequence of integer-valued functions $\left\{m_{i}\right\}$ where $m_{i}:K_{i}\to \mathbb{N}$.
   \end{enumerate}
   The construction is divided into three steps:

   $\it Step \ 1.$ Construct $K_{1}$, $\mu_{1}$, $m_{1}(\cdot )$ and $\gamma_{1}$.

   Note that $M^{\mathcal{P}}(t,\varepsilon,Z,\{F_{n}\},f,\mathbf{b})=+\infty$. Let
   $$H=\bigcup\left\{J \subset X:J \ \text{is open}, \ M^{\mathcal{P}}(t,\varepsilon,Z \cap J,\{F_{n}\},f,\mathbf{b})=0\right\}.$$
   Then by the separability of $X$, $H$ is a countable union of the open set $U$. It implies that $M^{\mathcal{P}}(t,\varepsilon,Z \cap H,\{F_{n}\},f,\mathbf{b})=0$. Let $Z'=Z\setminus H=Z \cap (X\setminus H)$. We first show that for any open set $J \subset X$, either $Z' \cap J=\emptyset $ or $M^{\mathcal{P}}(t,\varepsilon,Z' \cap J,\{F_{n}\},f,\mathbf{b})=0$.

   Suppose that $M^{\mathcal{P}}(t,\varepsilon,Z' \cap J,\{F_{n}\},f,\mathbf{b})=0$. Since $Z=Z' \cup (Z \cap H)$, we have
   $$M^{\mathcal{P}}(t,\varepsilon,Z \cap J,\{F_{n}\},f,\mathbf{b}) \leq M^{\mathcal{P}}(t,\varepsilon,Z' \cap J,\{F_{n}\},f,\mathbf{b})+M^{\mathcal{P}}(t,\varepsilon,Z \cap H,\{F_{n}\},f,\mathbf{b})=0.$$
   Thus $J \subset H$. It follows that $Z' \cap J=\emptyset$. Since
   $$M^{\mathcal{P}}(t,\varepsilon,Z,\{F_{n}\},f,\mathbf{b}) \leq M^{\mathcal{P}}(t,\varepsilon,Z \cap H,\{F_{n}\},f,\mathbf{b}) + M^{\mathcal{P}}(t,\varepsilon,Z',\{F_{n}\},f,\mathbf{b})$$
   and $M^{\mathcal{P}}(t,\varepsilon,Z \cap J,\{F_{n}\},f,\mathbf{b})=0$, it holds that
   $$M^{\mathcal{P}}(t,\varepsilon,Z',\{F_{n}\},f,\mathbf{b})= M^{\mathcal{P}}(t,\varepsilon,Z,\{F_{n}\},f,\mathbf{b})=+\infty.$$
   Thus $M^{\mathcal{P}}(s,\varepsilon,Z',\{F_{n}\},f,\mathbf{b})=+\infty$.

   Using Lemma \ref{l3}, we can find a finite set $K_{1} \subset Z'$ and an integer-valued function $m_{1}(x)$ on $K_{1}$ such that the collection $\{ \overline{B}_{F_{m_{1}(x)}}(x,\varepsilon) \}_{x \in K_{1}}$ is disjoint and
   $$\sum_{x \in K_{1}}e^{b(|F_{m_{1}(x)}|)s+f_{F_{m_{1}(x)}}(x)} \in (1,2).$$
   Define
   $$\mu_{1}=\sum_{x \in K_{1}}e^{b(|F_{m_{1}(x)}|)s+f_{F_{m_{1}(x)}}(x)} \delta_{x},$$
   where $\delta_{x}$ denotes the Dirac measure at $x$. Take $\gamma_{1}>0$ small enough such that for any $Z \in \overline{B}(x,\gamma_{1})$, we have
   \begin{equation}
      (\overline{B}(z,\gamma_{1}) \cup  \overline{B}_{F_{m_{1}(x)}}(z,\varepsilon)) \cap (\bigcup_{y \in K_{1} \setminus \{x\}}\overline{B}(y,\gamma_{1}) \cup  \overline{B}_{F_{m_{1}(y)}}(y,\varepsilon))=\emptyset.\label{eq:4.3}
   \end{equation}
   Since $K_{1} \subset Z'$, we have
   $$M^{\mathcal{P}}(t,\varepsilon,Z \cap B(x,\gamma_{1}/4),\{F_{n}\},f,\mathbf{b}) \geq M^{\mathcal{P}}(t,\varepsilon,Z' \cap B(x,\gamma_{1}/4),\{F_{n}\},f,\mathbf{b})>0,$$
   for any $x \in K_{1}$.

   $\it Step \ 2.$ Construct $K_{2}$, $\mu_{2}$, $m_{2}(\cdot )$ and $\gamma_{2}$.

   By \eqref{eq:4.3}, the elements of the family of balls $\{ \overline{B}(x,\gamma_{1})\}_{x \in K_{1}}$ are pairwise disjoint. For each $x \in K_{1}$, since $M^{\mathcal{P}}(t,\varepsilon,Z \cap B(x,\gamma_{1}/4),\{F_{n}\},f,\mathbf{b})>0$, we can construct a finite set $E_{2}(x) \subset Z \cap B(x,\gamma_{1}/4)$ and integer-valued function
   $$m_{2}(x):E_{2}(x) \to \mathbb{N} \cap [\text{max}\{m_{1}(y):y \in K_{1}\},+\infty)$$
   such that
   \begin{enumerate}
   \item[(a)] $M^{\mathcal{P}}(t,\varepsilon,Z \cap J,\{F_{n}\},f,\mathbf{b})>0$ for each open set $J$ with $J \cap E_{2}(x)\neq \emptyset$;
   \item[(b)] The elements in $\left \{ B_{F_{m_{2}(y)}}(y,\varepsilon ) \right \}_{y \in E_{2}(x)} $ are disjoint and
      $$ \mu_{1}(\{x\})<\sum_{y \in E_{2}(x)}e^{-b(|F_{m_{2}(y)}|)s+f_{F_{m_{2}(y)}}(y)}<(1+2^{-2})\mu_{1}(\{x\}). $$
   \end{enumerate}
   To see it, we fix $x \in K_{1}$ and denote $F=Z \cap B(x,\gamma_{1}/4)$. Let
   $$H_{x}=\bigcup \left\{ J \subset X :J \ \text{is open},\ M^{\mathcal{P}}(t,\varepsilon,F \cap J,\{F_{n}\},f,\mathbf{b})=0 \right\}.$$
   Set $F'=F \setminus H_{x}$. Then by Step 1, we can show that
   $$M^{\mathcal{P}}(t,\varepsilon,F',\{F_{n}\},f,\mathbf{b})=M^{\mathcal{P}}(t,\varepsilon,F,\{F_{n}\},f,\mathbf{b})>0$$
   and $M^{\mathcal{P}}(t,\varepsilon,F' \cap J,\{F_{n}\},f,\mathbf{b})>0$ for any open set $J$ with $J \cap F' \neq \emptyset$. Since $s<t$, we have
   $M^{\mathcal{P}}(s,\varepsilon,F',\{F_{n}\},f,\mathbf{b})=+\infty$.

   Using Lemma \ref{l3} again, we can find a finite set $E_{2} \subset F'$ and a map
   $$m_{2}(x):E_{2}(x) \to \mathbb{N} \cap [\text{max}\{m_{1}(y):y \in K_{1}\},+\infty)$$
   so that (b) holds. If $J \cap E_{2}(x) \neq \emptyset$ and $J$ is an open set, then $J \cap F' \neq \emptyset$. Hence
   $$M^{\mathcal{P}}(t,\varepsilon,Z \cap J,\{F_{n}\},f,\mathbf{b}) \geq M^{\mathcal{P}}(t,\varepsilon,F'\cap J,\{F_{n}\},f,\mathbf{b})>0.$$
   Thus (a) holds.
   Since the elements of the family $\{ \overline{B}(x,\gamma_{1})\}_{x \in K_{1}}$ are pairwise disjoint, $E_{2}(x) \cap E_{2}(y)=\emptyset$ for different points $x,y \in K_{1}$. Define $K_{2}=\bigcup_{x \in K_{1}}E_{2}(x)$ and
   $$\mu_{2}=\sum_{x \in K_{2}}e^{b(|F_{m_{2}(x)}|)s+f_{F_{m_{2}(x)}}(x)} \delta_{x}.$$
   By \eqref{eq:4.3} and (b), the elements in $\{ \overline{B}_{F_{m_{2}(x)}}(x,\varepsilon) \}_{x \in K_{2}}$ are pairwise disjoint. Hence we can take $\gamma_{2} \in (0,\gamma_{1}/4)$ small enough such that for any  $x \in K_{2}$ and the function $z:K_{2} \to X$ with $d(x,z(x))\leq \gamma_{2}$, we have
   $$(\overline{B}(z(x),\gamma_{2}) \cup  \overline{B}_{F_{m_{2}(x)}}(z(x),\varepsilon)) \cap (\bigcup_{y \in K_{2} \setminus \{x\}}\overline{B}(z(y),\gamma_{2}) \cup  \overline{B}_{F_{m_{2}(y)}}(z(y),\varepsilon))=\emptyset.$$
   Since $x \in K_{2}$, there exists $y \in K_{1}$ such that $x \in E_{2}(y)$. By (a), we have
   $$M^{\mathcal{P}}(t,\varepsilon,Z \cap B(x,\gamma_{2}/4),\{F_{n}\},f,\mathbf{b})>0,$$
   for each $x \in K_{2}$.

   $\it Step \ 3.$ Assume that $K_{i}$, $\mu_{i}$, $m_{i}(\cdot )$ and $\gamma_{i}$ have been constructed, for $i=1,2,\cdots,p$. In particular, suppose that for any $x \in K_{p}$ and the function $z:K_{p} \to X$ with $d(x,z(x))<\gamma_{p}$,
   \begin{equation}
      (\overline{B}(z(x),\gamma_{p}) \cup  \overline{B}_{F_{m_{p}(x)}}(z(x),\varepsilon)) \cap (\bigcup_{y \in K_{p} \setminus \{x\}}\overline{B}(z(y),\gamma_{p}) \cup  \overline{B}_{F_{m_{p}(y)}}(z(y),\varepsilon))=\emptyset \label{eq:4.4}
   \end{equation}
   and $M^{\mathcal{P}}(t,\varepsilon,Z \cap B(x,\gamma_{p}/4),\{F_{n}\},f,\mathbf{b})>0$. We shall construct $K_{p+1}$, $\mu_{p+1}$, $m_{p+1}(\cdot )$ and $\gamma_{p+1}$ in a way being similar to Step 2.

   Note that the elements in $\{ \overline{B}(x,\gamma_{p}) \}_{x \in K_{p}}$ are pairwise disjoint. Since
   $$M^{\mathcal{P}}(t,\varepsilon,Z \cap B(x,\gamma_{p}/4),\{F_{n}\},f,\mathbf{b})>0,$$
    for each $x \in K_{p}$,  we can construct as in Step 2, a finite set
   $$E_{p+1}(x) \subset Z \cap B(x,\gamma_{p}/4)$$
   and an integer-valued function
   $$m_{p+1}(x):E_{p+1}(x) \to \mathbb{N} \cap [\text{max}\{m_{p}(y):y \in K_{p}\},+\infty),$$
   such that
   \begin{itemize}
   \item[(c)] $M^{\mathcal{P}}(t,\varepsilon,Z \cap J,\{F_{n}\},f)>0$, for each open set $J$ with $J \cap E_{p+1}(x)\neq \emptyset$.
   \item[(d)] The elements in $\left \{ B_{F_{m_{p+1}(y)}}(y,\varepsilon ) \right \}_{y \in E_{p+1}(x)} $ are disjoint, and
      $$ \mu_{p}(\{x\})<\sum_{y \in E_{p+1}(x)}e^{-b(|F_{m_{p+1}(y)}|)s+f_{F_{m_{p+1}(y)}}(y)}<(1+2^{-p-1})\mu_{p}(\{x\}). $$
   \end{itemize}
   It is easy to see that $E_{p+1}(x) \cap E_{p+1}(y)=\emptyset$ for any $x,y \in K_{p}$ with $x \neq y$. Define
   $$K_{p+1}=\bigcup_{x \in K_{p}}E_{p+1}(x)$$ and
   $$\mu_{p+1}=\sum_{x \in K_{p+1}}e^{b(|F_{m_{p+1}(x)}|)s+f_{F_{m_{p+1}(x)}}(x)} \delta_{x}.$$
   By \eqref{eq:4.4} and (d), the elements in $\{ \overline{B}_{F_{m_{p+1}(x)}}(x,\varepsilon) \}_{x \in K_{p+1}}$ are pairwise disjoint. Hence we can take $\gamma_{p+1} \in (0,\gamma_{p}/4)$ small enough such that for any  $x \in K_{p+1}$ and the function $z:K_{p+1} \to X$ with $d(x,z(x))\leq \gamma_{p+1}$, we have
   $$(\overline{B}(z(x),\gamma_{p+1}) \cup  \overline{B}_{F_{m_{p+1}(x)}}(z(x),\varepsilon)) \cap (\bigcup_{y \in K_{p+1} \setminus \{x\}}\overline{B}(z(y),\gamma_{p+1}) \cup  \overline{B}_{F_{m_{p+1}(y)}}(z(y),\varepsilon))=\emptyset.$$
   Since $x \in K_{p+1}$, there exists $y \in K_{p}$ such that $x \in E_{p+1}(y)$. Thus by (c),
   $$M^{\mathcal{P}}(t,\varepsilon,Z \cap B(x,\gamma_{p+1}/4),\{F_{n}\},f,\mathbf{b})>0,$$
   for each $x \in K_{p+1}$.

   We summarize their properties as follow:
   \begin{itemize}
   \item[(e)] For each $i$, the family $\mathcal{F}_{i}=\{\overline{B}(x,\gamma_{i}):x \in K_{i}\}$ is disjoint. For every $B \in \mathcal{F}_{i+1}$, there exists $x \in K_{i}$ such that $B \subset \overline{B}(x,\gamma_{i}/2);$
   \item[(f)] For each $x \in K_{i}$ and $x' \in \overline{B}(x,\gamma_{i})$,
      \begin{equation}
         \overline{B}_{F_{m_{i}(x)}}(x',\varepsilon)\cap \bigcup_{y \in K_{i} \setminus \{x\}}\overline{B}(y,\gamma_{i})= \emptyset \label{eq:4.5}
      \end{equation}
      and
      \begin{equation}
         \mu_{i}(\overline{B}(x,\gamma_{i}))=e^{-b(|F_{m_{i}(x)}|)s+f_{F_{m_{i}(x)}}(x)} \leq \sum_{y \in E_{i+1}(x)}e^{-b(|F_{m_{i+1}(y)}|)s+f_{F_{m_{i+1}(y)}}(y)} \leq (1+2^{-i-1})\mu_{i}(\overline{B}(x,\gamma_{i})), \label{eq:4.6}
      \end{equation}
      where $E_{i+1}(x)=\overline{B}(x,\gamma_{i})\cap K_{i+1}$.
   \end{itemize}
   By \eqref{eq:4.6}, for any $F_{i} \in \mathcal{F}_{i}$, we have
   $$\begin{aligned}
      \mu(F_{i}) \leq \mu_{i+1}(F_{i})&=\sum_{F \in \mathcal{F}_{i+1}:F \subset F_{i}}\mu_{i+1}(F)\\
      &\leq \sum_{F \in \mathcal{F}_{i+1}:F \subset F_{i}}(1+2^{-i-1})\mu_{i}(F)\\
      &=(1+2^{-i-1})\sum_{F \in \mathcal{F}_{i+1}:F \subset F_{i}}\mu_{i}(F)\\
      &\leq (1+2^{-i-1})\mu_{i}(F_{i}).
   \end{aligned}$$
   Using the above inequality repeatedly, we have for any $j>i$,
   \begin{equation}
      \mu_{i}(F_{i}) \leq \mu_{j}(F_{i}) \leq \prod_{n=i+1}^{j}(1+2^{-n})\mu_{i}(F_{i})\leq C\mu_{i}(F_{i}), \ \forall F_{i}\in \mathcal{F}_{i}, \label{eq:4.7}
   \end{equation}
   where $C=\prod_{n=i+1}^{j}(1+2^{-n})<+\infty$.

   Let $\hat{\mu}$ be a limit point of $\{\mu_{i}\}$ in a weak$^{*}$ topology. Let
   $$K^{*}=\bigcap_{n=1}^{\infty}\overline{\bigcup_{i\geq n}K_{i}}=\lim_{n \to +\infty}\overline{\bigcup_{i\geq n}K_{i}}.$$
   Then $\hat{\mu}$ is supported on $K^{*}$. $K^{*}\subset Z$ and for any $i \in \mathbb{N}$, $K^{*}\subset \bigcup_{x \in K_{i}}B(x,\gamma_{i})$. By \eqref{eq:4.7},
   $$\begin{aligned}
      e^{-b(| F_{m_{i}(x)}|)s+f_{F_{m_{i}(x)}}(x) }&=\mu_{i}(\overline{B}(x,\gamma_{i})) \leq \hat{\mu}(\overline{B}(x,\gamma_{i}))\\
      &\leq C\mu_{i}(\overline{B}(x,\gamma_{i})) =C e^{-b(| F_{m_{i}(x)}|)s+f_{F_{m_{i}(x)}}(x) }, \ \forall x \in K_{i}.
   \end{aligned}$$
   In particular,
   $$1 \leq \sum_{x \in K_{1}}\mu_{1}(B(x,\gamma_{1}))\leq \sum_{x \in K_{1}}\hat{\mu}(B(x,\gamma_{1}))=\hat{\mu}(K^{*})\leq \sum_{x \in K_{1}}C\mu_{1}(B(x,\gamma_{1}))\leq 2C.$$
   By \eqref{eq:4.5}, for every $x \in K_{i}$ and $x'\in \overline{B}(x,\gamma_{i})$,
   $$\hat{\mu} (\overline{B}_{F_{m_{i}(x)}}(x',\varepsilon)) \leq \hat{\mu}(\overline{B}(x,\gamma_{i})) \leq C e^{-b(| F_{m_{i}(x)}|)s+f_{F_{m_{i}(x)}}(x) }.$$
   For each $x' \in K^{*}$ and $i \in \mathbb{N}$, there exists $x\in K_{i}$ such that $x'\in \overline{B}(x,\gamma_{i})$. Thus
   $$\hat{\mu} (\overline{B}_{F_{m_{i}(x)}}(x',\varepsilon)) \leq C e^{-b(| F_{m_{i}(x)}|)s+f_{F_{m_{i}(x)}}(x) }.$$
   Let $\mu=\hat{\mu}/\hat{\mu}(K^{*})$. Then $\mu \in M(x)$ and $\mu(K^{*})=1$. Moreover, for all $x' \in K^{*}$, there exists a sequence $\{k_{i}\}_{i\geq 1}$ with $k_{i} \to +\infty$ such that
   $$\mu(B_{F_{k_{i}}}(x',\varepsilon)) \leq \frac{Ce^{-b(| F_{k_{i}}|)s+f_{F_{k_{i}}}(x') }}{\hat{\mu}(K^{*})}.$$
   It implies that
   $$\frac{-\log{\frac{\hat{\mu}(K^{*}) }{C} } -\log{\mu(B_{F_{k_{i}}}(x',\varepsilon ))}+f_{F_{k_{i}}}(x')}{b(| F_{k_{i}}|) } \geq s.$$
   Letting $k_{i} \to +\infty$, we get $\overline{P}_{\mu }(\left \{ F_{n} \right \},f,\mathbf{b} )>s$.

\begin{rem}
   Theorem \ref{T3} was proved by Dou et al. in \cite{Dou} where $f=0$ and $b(|F_{n}|)=|F_{n}|$ for all $n \in \mathbb{N}$. Later Ding et al. \cite{Ding} proved the theorem when $b(|F_{n}|)=|F_{n}| $ for all $n \in \mathbb{N}$.
\end{rem}

\section{Scaled packing pressures for the set of generic points} \label{sec5}

 We know that there always exists a $T$-invariant Borel probability measure for a TDS with a $\mathbb{Z}$-action. This consequence is not necessarily true for all groups $G$ acting on $X$, but a known result says that there always exists a $G$-invariant Borel probability measure when a group $G$ is amenable. For more detail of amenable group actions, one can refer to \cite{Da,Do,Fe}.

For amenable group actions, the generic point is defined based on measure and the F{\o}lner sequence. Let $\left\{F_{n}\right\}$ be a F{\o}lner sequence and $\mu \in M(X,G)$. A point $x \in X$ is called a \textit{generic point} of $\mu$ (with respect to $\left\{F_{n}\right\}$) if
$$\lim_{n \to +\infty} \frac{1}{\left | F_{n}  \right | } \sum_{g\in F_{n}}f(gx)=\int_{X}f(x) \ \mathrm{d}\mu, \ \forall f \in C(X,\mathbb{R}) .$$
The set of generic points of $\mu$ (with respect to $\left\{F_{n}\right\}$) is denoted by $X_{\mu}$.

If $\mu \in E(X,G)$ and $\{F_{n}\}$ is a tempered F{\o}lner sequence, then $\mu(X_{\mu})=1$. In fact, let $\{f_{i}\}$ be a countable dense subset of $C(X,\mathbb{R})$. We denote
$$X_{i}=\left\{x\in X:\lim_{n \to +\infty} \frac{1}{\left | F_{n}  \right | } \sum_{g\in F_{n}}f_{i}(gx)=\int_{X}f_{i}(x) \ \mathrm{d}\mu\right\}.$$
By the pointwise ergodic theorem, $\mu(X_{i})=1$. Hence $X_{\mu}=\bigcap_{i=1}^{\infty}X_{i}$ has a full measure.

In this section, our main result is the following theorem.
\begin{thm}\label{T4}
    Let $(X,G)$ be a TDS with a $G$-action and a F{\o}lner  sequence $\left\{F_{n}\right\}$ satisfying Condition \ref{con1}. Suppose that $f \in C(X,\mathbb{R})$, $\mu \in E(X,G)$, $\left\{F_{n}\right\}$ is a tempered F{\o}lner sequence in $G$ and
    \begin{equation} \label{q3}
    \lim\limits_{n \to \infty}\frac{|F_{n}|}{b(|F_{n}|)}=1,
    \end{equation} then we have
\begin{equation}
P^{P}(X_{\mu},\left\{F_{n}\right\},f,\mathbf{b})=h_{\mu}(X)+\int_{X} f \mathrm{d}\mu , \label{eq:5.1}
\end{equation}
where $h_{\mu}(X)$ denotes the measure-theoretic entropy.
\end{thm}
\begin{rem}
In 1973, Bowen \cite{Bowen} proved a variational principle of entropies on the set of generic points when $\mu$ is ergodic. Pesin and Pitskel \cite{Pe2} extended this result to topological pressure. Zhong and Chen \cite{Zh} proved equation \eqref{eq:5.1} where $G=\mathbb{Z}$ and $b(|F_{n}|)=|F_{n}|$ for all $n\in \mathbb{N}$. Ding et al. \cite{Ding} showed equation \eqref{eq:5.1} holds where $b(|F_{n}|)=|F_{n}|$ for all $n \in \mathbb{N}$.
\end{rem}
Let $F \in F(G)$ and $C \subset M(X)$. We denote
$$X_{F,C}=\{x \in X:\xi_{F}(x) \in C  \},$$
where $\xi_{F}(x)=\frac{1}{|F|}\sum_{g \in F}\delta_{gx}$ and $\delta_{x}$ is the Dirac measure at $x \in X$. A subset $E \subset X$ is said to be an \textit{$ (F,\varepsilon)$-separated} if for any $x\neq y \in E$, $d(hx,hy)>\varepsilon$ for any $ h \in F$.

To prove Theorem \ref{T4}, we need the following three lemmas.

\begin{lem} \label{l4} \cite[Lemma 4.2]{Zhang}
   Let $\{F_{n}\}$ be a F{\o}lner sequence and $\mu \in M(X,G)$. Suppose that $C \subset M(X)$ is a neighbourhood of $\mu$, $f \in C(X,\mathbb{R})$ and
   $$N_{f}(C;F_{n},\varepsilon):=\sup_{E}\sum_{x \in E}e^{f_{F_{n}}(x)},$$
   where the supremum is taken over all $(F_{n},\varepsilon)$-separated sets $E \subset X_{F_{n},C}$. Then
   $$\lim_{\varepsilon \to 0}\inf_{C}\limsup_{n \to +\infty}\frac{1}{|F_{n}|}\log{N_{f}(C;F_{n},\varepsilon)} \leq h_{\mu}(X)+\int_{X} f \mathrm{d}\mu.$$
\end{lem}

The following result holds without the property of ergodic and the condition \eqref{q3}.

\begin{lem} \label{l5}
    Let $(X,G)$ be a TDS with a $G$-action and a F{\o}lner  sequence $\left\{F_{n}\right\}$ satisfying Condition \ref{con1}. If $f \in C(X,\mathbb{R})$ and $\mu \in M(X,G)$, then
   $$P^{P}(X_{\mu},\left\{F_{n}\right\},f,\mathbf{b}) \leq h_{\mu}(X)+\int_{X} f \mathrm{d}\mu. $$
\end{lem}
\begin{proof}
   For any neighbourhood $C\subset M(X)$ of $\mu$ and $m \in \mathbb{N}$, let
$$X_{\mu}^{m}=\{x\in X_{\mu}: \xi_{F_{n}}(x)\in C,\forall n\geq m\}.$$
Then $X_{\mu}=\bigcup_{m=1}^{\infty}X_{\mu}^{m}$ and $X_{\mu}^{m} \subset X_{F_{n},C}$ for all $n \geq m$. Fix $m\in \mathbb{N}$. Let $\alpha<\beta<P^{P}(X_{\mu}^{m},\{F_{n}\},f,\mathbf{b})$. Then there exists $\varepsilon'>0$ such that $P^{P}(\varepsilon,X_{\mu}^{m},\{F_{n}\},f,\mathbf{b})>\beta$ for all $\varepsilon \in (0,\varepsilon')$. It follows that
$$M^{P}(\beta,\varepsilon,X_{\mu}^{m},\{F_{n}\},f,\mathbf{b})\geq M^{\mathcal{P}}(\beta,\varepsilon,X_{\mu}^{m},\{F_{n}\},f,\mathbf{b})=+\infty.$$
Thus
$$M^{P}(N,\beta,\varepsilon,X_{\mu}^{m},\{F_{n}\},f,\mathbf{b})=+\infty, \ \forall N \in \mathbb{N}.$$
Since $\mathbf{b}$ and $\left\{F_{n}\right\}$ satisfy Condition \ref{con1}, $\sum_{n=1}^{\infty}e^{b(|F_{n}|)(\alpha-\beta)}$ converges. Let $\sum_{n=1}^{\infty}e^{b(|F_{n}|)(\alpha-\beta)}=M$. Then for a given number $N\geq m$, we can find a finite or countable pairwise disjoint family $\left\{\overline{B}_{F_{n_{i}}} (x_{i},\varepsilon )\right\}_{i}$ such that $x_{i}\in X_{\mu}^{m}$, $n_{i}\geq N$ and
$$\sum_{i}e^{-\beta b(|F_{n_{i}}|)+f_{F_{n_{i}}}(x_{i})}>M.$$
For each $k\geq N$, let
$$X_{\mu}^{m,k}=\{x_{i}\in X_{\mu}^{m}:n_{i}=k\}.$$
Then
$$\sum_{k=N}^{\infty}(e^{-\beta b(|F_{k}|)}\sum_{x\in X_{\mu}^{m,k}}e^{f_{F_{k}}(x)})=\sum_{i}e^{-\beta b(|F_{n_{i}}|)+f_{F_{n_{i}}}(x_{i})}>M.$$
It is not hard to see that there exists $k\geq N$ with
$$\sum_{x\in X_{\mu}^{m,k}}e^{f_{F_{k}}(x)} \geq e^{b(|F_{k}|)\alpha}(1-e^{\alpha-\beta}).$$
Since $X_{\mu}^{m,k}$ is a ($F_{k},\varepsilon$)-separated set of $X_{F_{k},C}$, we have
$$N_{f}(C;F_{k},\varepsilon)\geq e^{b(|F_{k}|)\alpha}(1-e^{\alpha-\beta}).$$
Thus
$$\limsup_{n \to +\infty}\frac{1}{|F_{n}|}\log{N_{f}(C;F_{n},\varepsilon)} \geq \alpha.$$
Since $\varepsilon \in (0,\varepsilon')$ and $C \ni \mu$ are arbitrary, we have
$$\lim_{\varepsilon \to 0}\inf_{C}\limsup_{n \to +\infty}\frac{1}{|F_{n}|}\log{N_{f}(C;F_{n},\varepsilon)}\geq \alpha.$$
Using Lemma \ref{l4}, we have
$$h_{\mu}(X)+\int_{X} f \mathrm{d}\mu \geq \alpha.$$
Thus
$$h_{\mu}(X)+\int_{X} f \mathrm{d}\mu \geq P^{P}(X_{\mu}^{m},\{F_{n}\},f,\mathbf{b}).$$
By Proposition \ref{p2}, It follows that
$$h_{\mu}(X)+\int_{X} f \mathrm{d}\mu \geq P^{P}(X_{\mu},\{F_{n}\},f,\mathbf{b})=\sup_{m}P^{P}(X_{\mu}^{m},\{F_{n}\},f,\mathbf{b}).$$
\end{proof}
\begin{lem} \label{l6}
   Let $(X,G)$ be a TDS with a $G$-action and a F{\o}lner  sequence $\left\{F_{n}\right\}$. Then for any $f\in C(X)$ and any nonempty compact subset $K\subseteq X$,
   $$P^{B}(K,\{F_{n}\},f,\mathbf{b}) \geq \sup \left\{ \underline{P}_{\mu }(K,\left \{ F_{n} \right \},f,\mathbf{b} ):\mu \in M(X), \ \mu(K)=1 \right\}.$$
\end{lem}
\begin{proof}
   We first show that for $s\in \mathbb{R}$, if  $\underline{P}_{\mu }(x,\left \{ F_{n} \right \},f,\mathbf{b}) \geq s$ for all $x\in Z$ and $\mu(Z)\geq 0$,then $P^{B}(Z,\{F_{n}\},f,\mathbf{b})\geq s$.

   Fix $\beta >0$. For each $m \geq 1$, put
   $$Z_{m}=\left \{ x \in Z:\liminf_{n \to +\infty} \frac{-\log{\mu}(B_{F_{n}}(x,\varepsilon ))+f_{F_{n}} (x)}{b(|  F_{n} |)} > s-\beta ,\ \forall \varepsilon  \in \left(0,\frac{1}{m} \right] \right \}.$$
   Since \[\frac{-\log{\mu}(B_{F_{n}}(x,\varepsilon ))+f_{F_{n}} (x)}{b(|  F_{n} |)}\] increases when $\varepsilon$ decreases, if follows that
   $$Z_{m}=\left \{ x \in Z:\liminf_{n \to +\infty} \frac{-\log{\mu}(B_{F_{n}}(x,\varepsilon ))+f_{F_{n}} (x)}{b(|  F_{n} |)} > s-\beta ,\ \varepsilon = \frac{1}{m} \right \}.$$
   Then $Z_{m} \subset Z_{m+1}$ and $\bigcup_{m=1}^{\infty}Z_{m}=Z$. So by the continuity of the measure, we have
   $$\lim_{m \to \infty}\mu(Z_{m})=\mu(Z).$$
   Take $M \geq 1$ with $\mu(Z_{M})\geq \frac{1}{2}\mu(Z)$. For every $N \geq 1$, put\
   $$\begin{aligned}
      Z_{M,N}&=\left \{ x \in Z_{M}:\frac{-\log{\mu}(B_{F_{n}}(x,\varepsilon ))+f_{F_{n}} (x)}{b(|  F_{n} |)} > s-\beta ,\ \forall n\geq N, \varepsilon  \in \left(0,\frac{1}{M} \right] \right \}\\
      &=\left \{ x \in Z_{M}:\frac{-\log{\mu}(B_{F_{n}}(x,\varepsilon ))+f_{F_{n}} (x)}{b(|  F_{n} |)} > s-\beta ,\ \forall n\geq N, \varepsilon =\frac{1}{M} \right \}.
   \end{aligned}$$
   Thus $Z_{M,N}\subset Z_{M,N+1}$ and $\bigcup_{N=1}^{\infty}Z_{M,N}=Z_{M}$. Now we can find $N^{*}\geq 1$ such that $\mu(Z_{M,N^{*}})>\frac{1}{2}\mu(Z_{M})>0$. For every $x\in Z_{M,N^{*}}$, $n\geq N^{*}$, and $0<\varepsilon <\frac{1}{M}$, we have
   $$\mu(B_{F_{n}}(x,\varepsilon ))\leq e^{-b(|F_{n}|)(s -\beta)+f_{F_{n}}(x)}.$$
   Let $\mathcal{F}=\left\{ B_{F_{n_{i}}}(y_{i},\frac{\varepsilon}{2} ) \right\}_{i\geq 1}$ be an open cover of $Z_{M,N^{*}}$ such that $Z_{M,N^{*}}\subset \bigcup_{i=1}^{\infty} B_{F_{n_{i}}}(y_{i},\frac{\varepsilon}{2} )$ and
   $$Z_{M,N^{*}}\cap B_{F_{n_{i}}}(y_{i},\frac{\varepsilon}{2}) \neq \emptyset, \ n_{i}\geq N^{*}, \forall i\geq 1,\ \text{and} \ 0<\varepsilon < \frac{1}{M}.$$
   For each $i \geq 1$, there exists $x_{i}\in Z_{M,N^{*}}\cap B_{F_{n_{i}}}(y_{i},\frac{\varepsilon}{2})$. Hence we have
   $$B_{F_{n_{i}}}(y_{i},\frac{\varepsilon}{2}) \subset B_{F_{n_{i}}}(x_{i},\varepsilon)$$
   by the triangle inequality.
   It follows that
   $$\begin{aligned}
      \sum_{i\geq 1}e^{-b(|F_{n_{i}}|)(s -\beta)+f_{F_{n_{i}}}(y_{i},\frac{\varepsilon}{2})} &\geq \sum_{i\geq 1}e^{-b(|F_{n_{i}}|)(s -\beta)+f_{F_{n_{i}}}(x_{i})} \\
      &\geq \sum_{i\geq 1}\mu(B_{F_{n_{i}}}(x_{i},\varepsilon)) \geq \mu(Z_{M,N^{*}})\geq 0.
   \end{aligned}$$
   Thus
   $$M(N,s-\beta,\frac{\varepsilon}{2} ,Z_{M,N^{*}},\left\{F_{n}\right\},f,\mathbf{b}) \geq \mu(Z_{M,N^{*}})>0.$$
   This implies that $P^{B}(Z,\{F_{n}\},f,\mathbf{b})\geq P^{B}(Z_{M,N^{*}},\{F_{n}\},f,\mathbf{b})\geq s-\beta$. Since $\beta$ is arbitrary, we get $P^{B}(Z,\{F_{n}\},f,\mathbf{b})\geq s$.
   For any $\delta>0$ and $\mu \in M(X)$ with $\mu(K)=1$, the set
   $$K_{\delta}=\left\{x\in K:\underline{P}_{\mu }(x,\left \{ F_{n} \right \},f,\mathbf{b}) \geq \underline{P}_{\mu }(K,\left \{ F_{n} \right \},f,\mathbf{b} )-\delta \right\}$$
   has positive measure in $\mu$. Then by the above discussion, we have
   $$P^{B}(K_{\delta},\{F_{n}\},f,\mathbf{b}) \geq \underline{P}_{\mu }(K,\left \{ F_{n} \right \},f,\mathbf{b} )-\delta.$$
   Since $K_{\delta} \subset K$, we have $P^{B}(K,\{F_{n}\},f,\mathbf{b}) \geq \underline{P}_{\mu }(K,\left \{ F_{n} \right \},f,\mathbf{b} )-\delta$. The arbitrariness of $\delta$ and $\mu$ imply that
   $$P^{B}(K,\{F_{n}\},f,\mathbf{b}) \geq \sup \left\{ \underline{P}_{\mu }(K,\left \{ F_{n} \right \},f,\mathbf{b} ):\mu \in M(X), \ \mu(K)=1 \right\}.$$
\end{proof}
We now prove Theorem 5.1 according to the above lemmas.

\textit{Proof of Theorem 5.1}.
We shall show that
$$h_{\mu}(X)+\int_{X} f \mathrm{d}\mu \leq P^{P}(X_{\mu},\{F_{n}\},f,\mathbf{b}).$$
The theorem then follows from Lemma \ref{l5}.
By Lemma \ref{l6}, we have
$$P^{B}(X_{\mu},\{F_{n}\},f,\mathbf{b}) \geq \underline{P}_{\mu }(X_{\mu},\left \{ F_{n} \right \},f,\mathbf{b} ).$$
Since $\lim\limits_{n \to \infty}\frac{|F_{n}|}{b(|F_{n}|)}=1$, by Theorem 3.1 in \cite{Zheng}, it follows that
$$\begin{aligned}
   P^{B}(X_{\mu},\{F_{n}\},f,\mathbf{b})& \geq \int_{X_{\mu}} \lim_{\varepsilon  \to 0}\liminf_{n \to +\infty} \frac{-\log{\mu(B_{F_{n}}(x,\varepsilon ))}+f_{F_{n}}(x)}{b(|F_{n}|)} \mathrm{d}\mu \\
   & = \int_{X_{\mu}} \lim_{\varepsilon  \to 0}\liminf_{n \to +\infty} \frac{-\log{\mu(B_{F_{n}}(x,\varepsilon ))}+f_{F_{n}}(x)}{|F_{n}|} \mathrm{d}\mu \\
   & \geq \int_{X_{\mu}}\lim_{\varepsilon  \to 0}\liminf_{n \to +\infty} \frac{-\log{\mu(B_{F_{n}}(x,\varepsilon ))}}{|F_{n}|} \mathrm{d}\mu +\int_{X_{\mu}} \lim_{n \to +\infty} \frac{f_{F_{n}}(x)}{|F_{n}|} \mathrm{d}\mu \\
   & \geq h_{\mu}(X)+\int_{X_{\mu}} \lim_{n \to +\infty} \frac{f_{F_{n}}(x)}{|F_{n}|} \mathrm{d}\mu=h_{\mu}(X)+\int_{X} f \mathrm{d}\mu.
\end{aligned}$$
Then combining with Lemma \ref{l1}, we get the desired inequality.

\section*{\textbf{Acknowledgments}}%
The research is supported by NNSF of China (Grant No.12201120). The authors would
like to thank the anonymous referees for their valuable comments and suggestions.
\section*{\textbf{Conflict of interests}}%
On behalf of all authors, the corresponding author states that there is no conflict of interests.


\end{sloppypar}
\end{document}